\documentclass[rm]{birkmult}
\usepackage{amssymb,latexsym}

\mathchardef\varSigma="0106

\def\bbK{\mathbf{K}}

\def\bbR{\mathbf{R}}
\def\bbC{\mathbf{C}}
\def\rto{\mathbf{R}\hskip-.5pt^2}

\def\hyp{\hskip.5pt\vbox
{\hbox{\vrule width3ptheight0.5ptdepth0pt}\vskip2.2pt}\hskip.5pt}

\def\tb{T\hskip-.3pt\bs}
\def\tab{{T\hskip.2pt^*\hskip-2.3pt\bs}}
\def\tayb{{T_{\nnh y}^*\hskip-1.9pt\bs}}
\def\tyb{{T\hskip-2pt_y\hskip-.2pt\bs}}

\def\line{\Lambda}
\def\plane{\Pi}
\def\ctr{\varOmega}

\def\ust{U\hskip-2pt_*}
\def\bl{\mathcal{L}}

\def\bv{\mathcal{V}}
\def\bp{\mathcal{P}}

\def\mpp{\hbox{$-$\hskip1pt$+$\hskip1pt$+$}}
\def\mmpp{\hbox{$-$\hskip1pt$-$\hskip1pt$+$\hskip1pt$+$}}

\def\x{\mathfrak{h}}

\def\T{\varTheta}
\def\t{\theta}

\def\bs{\varSigma}
\def\hs{\hskip.7pt}
\def\hh{\hskip.4pt}
\def\hn{\hskip-.35pt}
\def\nh{\hskip-.7pt}
\def\nnh{\hskip-1pt}

\def\vg{\varGamma}

\def\vs{vector space}

\def\vf{vector field}

\def\psr{pseu\-\hbox{do\hs-}Riem\-ann\-i\-an}
\def\inv{-in\-var\-i\-ant}

\def\diml{-di\-men\-sion\-al}

\def\om{\alpha}%{\hh\omega\hh}
\def\zy{\hh\omega}
\def\sym{\sigma}

\def\mf{manifold}
\def\mfd{-man\-i\-fold}

\def\ctb{cotangent bundle}
\def\prc{pseu\-\hbox{do\hs-}Riem\-ann\-i\-an metric}

\def\kerd{\text{\rm Ker}\hskip2.7ptd}

\def\ts{total space}
\def\pmb{\pi}
 \newtheorem{thm}{Theorem}[section]
 \newtheorem{cor}[thm]{Corollary}
 \newtheorem{lem}[thm]{Lemma}
 
 \theoremstyle{definition}
 
 \theoremstyle{remark}
 \newtheorem{rem}[thm]{Remark}
 \newtheorem{ex}[thm]{Example}
 \numberwithin{equation}{section}

\begin{document} 

\title[Connections with skew-sym\-met\-ric Ric\-ci tensor]{Connections with 
skew-sym\-met\-ric Ric\-ci tensor on surfaces}

\author[A. Derdzinski]{Andrzej Derdzinski} 
\address{Department of Mathematics\\
The Ohio State University\\
Columbus, OH 43210\\
USA}

\email{andrzej@math.ohio-state.edu} 


%\date{January 31, 2008}
\dedicatory{\ \ \ }
%%% ----------------------------------------------------------------------

%\urladdr{http://www.math.ohio-state.edu/\~{}andrzej}

\begin{abstract}
Some known results on torsionfree connections with 
skew-symmet\-ric Ricci tensor on surfaces are extended to connections with 
torsion, and Wong's canonical coordinate form of such connections is 
simplified.
\end{abstract} 

%%% ----------------------------------------------------------------------
\maketitle
%%% ----------------------------------------------------------------------
%\tableofcontents

%\voffset=-17pt\hoffset=0pt % for arxiv

\section{Introduction}\label{intro}
This paper generalizes some results concerning the situation where
\begin{equation}\label{srt}
\begin{array}{l}
\nabla\,\mathrm{\ is\ a\ connection\ on\ a\ surface\ }\,\bs,
\mathrm{\hskip5ptand\ the\ Ric\-ci}\\
\mathrm{tensor\ }\,\rho\,\mathrm{\ of\ }\,\nabla\hn\,\mathrm{\ is\ 
skew}\hyp\mathrm{sym\-met\-ric\ at\ every\ point.}
\end{array}
\end{equation}
What is known about condition (\ref{srt}) can be summarized as follows. 
Nor\-den \cite{norden-48}, \cite[\S89]{norden-50} showed that, for a 
tor\-sion\-free connection $\,\nabla\hn\,$ on a surface, skew-sym\-me\-try 
of the Ric\-ci tensor is equivalent to flatness of the connection 
obtained by projectivizing $\,\nabla\nh$, and implies the existence of a 
frac\-tion\-al-lin\-e\-ar first integral for the geodesic equation. Wong 
\cite[Theorem~4.2]{wong} found three coordinate expressions which, locally, 
represent all tor\-sion\-free connections $\,\nabla\hn\,$ with (\ref{srt}) 
such that $\,\rho\ne0$ everywhere in $\,\bs$. Ko\-wal\-ski, O\-poz\-da 
and Vl\'a\-\v sek \cite{kowalski-opozda-vlasek-00} used an approach different 
from Wong's to classify, locally, all tor\-sion\-free connections 
$\,\nabla\hn\,$ satisfying (\ref{srt}) that are also locally homogeneous, 
while in \cite{kowalski-opozda-vlasek-99} they proved that, for 
real-an\-a\-lyt\-ic tor\-sion\-free connections $\,\nabla\hn\,$ with 
(\ref{srt}), third-or\-der cur\-va\-ture\hskip.5pt-ho\-mo\-ge\-ne\-i\-ty 
implies local homogeneity (but one cannot replace the word `third' with 
`second'). Bla\-\v zi\'c and Bo\-kan \cite{blazic-bokan} showed that the torus 
$\,T\hh^2$ is the only closed surface $\,\bs\,$ admitting both a 
tor\-sion\-free connection $\,\nabla\hn\,$ with (\ref{srt}) and a 
$\,\nabla\nh$-par\-al\-lel almost complex structure. Garc\'\i a\hh-R\'\i o, 
Kupeli, V\'azquez\hh-Abal \hbox{and\ V\'azquez\hh-}Lorenzo 
\cite{garcia-rio-kupeli-vazquez-abal-vazquez-lorenzo} 
proved that connections $\,\nabla\hn\,$ as in (\ref{srt}) are equivalently 
characterized both by being the so-call\-ed {\it af\-fine Osser\-man 
connections\/} on surfaces 
\cite[Theorem~4]{garcia-rio-kupeli-vazquez-abal-vazquez-lorenzo}, and, in the 
tor\-sion\-free case, by the four\diml\ Osser\-man property of the Riemann 
extension metric $\,g\nnh^\nabla$ 
\cite[Theorem~4]{garcia-rio-kupeli-vazquez-abal-vazquez-lorenzo}. They also 
showed that, if such $\,\nabla\hn\,$ is tor\-sion\-free and $\,\rho\ne0$ 
everywhere, then $\,g\nnh^\nabla$ is a 
cur\-va\-ture\hskip.5pt-ho\-mo\-ge\-ne\-ous self-du\-al Ric\-ci-flat Walker 
metric of Pe\-trov type III with the metric signature $\,(\mmpp)\,$ on the 
four\mfd\ $\,\tab\,$ 
\cite[Theorem~9]{garcia-rio-kupeli-vazquez-abal-vazquez-lorenzo}, cf.\ 
\cite[Remark 2.1]{diaz-ramos-garcia-rio-vazquez-lorenzo}. An\-der\-son and 
Thomp\-son \cite[pp.\ 104--107]{anderson-thompson} proved that, among 
tor\-sion\-free connections $\,\nabla\hn\,$ on surfaces, those with 
(\ref{srt}) are characterized by the existence, locally in $\,\tb$, of a 
frac\-tion\-al-lin\-e\-ar La\-grang\-i\-an for which the geodesics of 
$\,\nabla\hs$ are the solutions of the Eu\-ler-La\-grange equations. Bo\-kan, 
Matzeu and Raki\'c \cite{bokan} -- \cite{bokan-matzeu-rakic-06} studied 
connections with skew-sym\-met\-ric Ric\-ci tensor on 
high\-er-di\-men\-sion\-al manifolds.

The results of this paper begin with Section~\ref{tccs}, where we obtain the 
conclusion of Bla\-\v zi\'c and Bo\-kan \cite{blazic-bokan} without assuming 
the existence of a $\,\nabla\nh$-par\-al\-lel almost complex structure, while 
allowing $\,\nabla\hn\,$ to have torsion. 

Next, in Section~\ref{wong}, we strengthen Wong's theorem 
\cite[Theorem~4.2]{wong} by reducing the three cases to just one, and removing 
the assumption that $\,\rho\ne0$.

In Sections~\ref{lclg} --~\ref{glhc} we extend the theorem of Ko\-wal\-ski, 
O\-poz\-da and Vl\'a\-\v sek \cite{kowalski-opozda-vlasek-00} to connections with 
torsion, proving that a locally homogeneous connection on a surface having 
skew-sym\-met\-ric Ric\-ci tensor must be locally equivalent to a 
left-in\-var\-i\-ant connection on a Lie group (Theorem~\ref{genrl}). Note 
that the last conclusion is also true for all locally homogeneous 
tor\-sion\-free connections on surfaces except the Le\-vi-Ci\-vi\-ta 
connection of the standard sphere, as one easily verifies using O\-poz\-da's 
classification of such connections \cite[Theorem 1.1]{opozda}. See 
Remark~\ref{lhlie}.

Sections~\ref{nthm} and~\ref{lagr} generalize some of Nor\-den's results 
\cite{norden-50} to connections with torsion. In Section~\ref{lagr} we also 
give a proof of An\-der\-son and Thomp\-son's theorem 
\cite{anderson-thompson} based on the Ham\-il\-ton\-i\-an formalism.

Finally, Section~\ref{drgr} describes a class of examples of Ric\-ci-flat 
Walker four\mfd s which includes those constructed by Garc\'\i a\hh-R\'\i o, 
Kupeli, V\'azquez\hh-Abal and V\'azquez\hh-Lorenzo 
\cite[Theorem~9]{garcia-rio-kupeli-vazquez-abal-vazquez-lorenzo}. The 
generalization arises by the use of a more general type of Riemann extensions. 
However, our examples are not new: they first appeared, in a different form, 
in Theorem~3.1(ii.3) of 
D\'\i az\hh-Ramos, \hbox{Garc\'\i a\hh}\-R\'\i o and V\'azquez\hh-Lorenzo's 
paper 
\cite{diaz-ramos-garcia-rio-vazquez-lorenzo}. In addition, as we point out in 
Section~\ref{drgr}, Theorem~3.1(ii.3) of 
\cite{diaz-ramos-garcia-rio-vazquez-lorenzo} states in coordinate language 
that locally, up to isometries, this larger class of examples consists 
precisely of all cur\-va\-ture\hskip.5pt-ho\-mo\-ge\-ne\-ous self-du\-al 
Ric\-ci-flat Walker $\,(\mmpp)\,$ metrics of Pe\-trov type III.

\section{Preliminaries}\label{prel}
By a `manifold' we always mean a connected manifold. All manifolds, bundles, 
their sections and sub\-bun\-dles, connections and mappings, including bundle 
morphisms, are assumed to be $\,C^\infty\nnh$-dif\-fer\-en\-ti\-a\-ble, while 
a bundle morphism, by definition, operates between two bundles with the same 
base manifold, and acts by identity on the base. For the exterior product of 
$\,1$-forms $\,\xi,\eta\,$ and a $\,2$-form $\,\om\,$ on a manifold, the 
exterior derivative of $\,\xi$, and any tangent vector fields $\,u,v,w$, we 
have
\begin{equation}\label{dxi}
\begin{array}{rl}
\mathrm{a)}\hskip7pt&
(\xi\wedge\eta)(\nh u,v)\,=\,\xi(u)\eta(v)\,-\,\eta(u)\hs\xi(v)\hs,\\
\mathrm{b)}\hskip7pt&
(\xi\wedge\om)(\nh u,v,w)\,=\,\xi(u)\hs\om(v,w)\,+\,\xi(v)\hs\om(w,u)
+\,\xi(w)\hs\om(\nh u,v)\hs,\\
\mathrm{c)}\hskip7pt&
(d\hs\xi)(\nh u,v)\,=\,d_u[\hs\xi(v)]\,-\,d_v[\hs\xi(u)]\,-\,\xi([u,v])\hs.
\end{array}
\end{equation}
Our sign convention about the curvature tensor $\,R=R^\nabla$ of a connection 
$\,\nabla\hn\,$ in a real or complex vector bundle $\,\bv\,$ over a manifold 
$\,\bs\,$ is
\begin{equation}\label{cur}
R\hh(\nh u,v)\hs\psi\hskip7pt=\hskip7pt\nabla_{\hskip-2.2ptv}\nabla_{\hskip-2.2ptu}\psi\,
-\,\nabla_{\hskip-2.2ptu}\nabla_{\hskip-2.2ptv}\psi\,+\,\nabla_{\nh[u,v]}\psi
\end{equation}
for sections $\,\psi\,$ of $\,\bv\,$ and \vf s $\,u,v\,$ tangent to $\,\bs$. 
We then denote by
\begin{equation}\label{ope}
\mathrm{i)}\hskip9ptR\hh(\nh u,v):\bv\to\bv\hs,\hskip29pt\mathrm{ii)}\hskip9pt
\ctr(\nh u,v)=\mathrm{tr}\hskip2pt[R\hh(\nh u,v)]:\bs\to\bbK\hs,
\end{equation}
the bundle morphism sending any $\,\psi\,$ to $\,R\hh(\nh u,v)\hs\psi$, and, 
respectively, its pointwise trace, $\,\bbK\,$ being the scalar field ($\bbR\,$ 
or $\,\bbC$). Thus, $\,\ctr\,$ is a $\,\bbK$-val\-ued $\,2$-form on $\,\bs$. 
If, in addition, $\,\bv\,$ is a line bundle (of fibre dimension $\,1$), then 
$\,R\hh(\nh u,v)=\ctr(\nh u,v)$, that is, the morphism $\,R\hh(\nh u,v)\,$ 
acts via multiplication by the $\,\bbK$-val\-ued function $\,\ctr(\nh u,v)$, 
and we call $\,\ctr\,$ the {\it curvature form\/} of $\,\nabla$.

The torsion tensor $\,\hh\T\nnh\,$ of a connection $\,\nabla\hn\,$ on a 
manifold $\,\bs\,$ is characterized by 
$\,\T(v,w)=\nabla_{\hskip-2.2ptv}w-\nabla_{\hskip-2.7ptw}v-[v,w]$, for 
vector fields $\,v,w\,$ tangent to $\,\bs$. If $\,\bs\,$ is a surface, 
$\,\T\,$ is completely determined by the {\it torsion form\/} $\,\t$, which is 
the $\,1$-form with $\,\t(v)=\hh\mathrm{tr}\hskip2.7pt\T(v,\,\cdot\,)$. In 
fact, $\,\T=\t\wedge\mathrm{Id}\hh$, that is, 
$\,\T(v,w)=\t(v)\hh w-\t(w)\hh v$.

For a connection $\,\nabla\hn\,$ on a surface, its torsion form $\,\t$, and any 
$\,1$-form $\,\xi$,
\begin{equation}\label{dxt}
d\hs\xi\,\,=\,\hs\nabla\xi\,\,-\,\hs(\nabla\xi)^*\hs+\,\,\t\wedge\hs\xi
\end{equation}
in the sense that $\,(d\hs\xi)(\nh u,v)=(\nabla_{\hskip-2.2ptu\hs}\xi)(v)
-(\nabla_{\hskip-2.2ptv\hs}\xi)(u)+\t(u)\hs\xi(v)-\t(v)\hs\xi(u)\,$ 
whenever $\,u,v\,$ are tangent vector fields. This is clear from (\ref{dxi}) 
and the last paragraph.
\begin{rem}\label{dtbdl}The {\it determinant bundle\/} of a real\hs/com\-plex 
vector bundle $\,\bv\,$ of fibre dimension $\,m\,$ is its highest 
real\hs/com\-plex exterior power $\,\det\hskip1pt\bv=\bv^{\wedge m}\nnh$. For 
any connection $\,\nabla\hn\,$ in $\,\bv$, the $\,2$-form $\,\ctr\,$ defined by 
(\ref{ope}.ii) is the curvature form of the connection in the line bundle 
$\,\det\hskip1pt\bv\,$ induced by $\,\nabla$.
\end{rem}
\begin{rem}\label{cohom}In view of the Bianchi identity and 
Remark~\ref{dtbdl}, the form $\,\ctr\,$ in (\ref{ope}.ii) is always closed. 
Its cohomology class $\,[\hh\ctr\hs]\in H^2(\bs,\bbK)\,$ does not depend on 
the choice of the connection $\,\nabla$. (This is again immediate from 
Remark~\ref{dtbdl}: two connections in the line bundle $\,\det\hskip1pt\bv\,$ 
differ by a $\,\bbK$-val\-ued $\,1$-form $\,\xi\,$ on $\,\bs$, and so, by 
(\ref{cur}) and (\ref{dxi}.c), their curvature forms differ by $\,-\hs d\xi$.) 
Specifically, $\,[\hh\ctr\hs]$ equals $\,2\pi\,$ times $\,c_1(\bv)\,$ when 
$\,\bbK=\bbC\hs$, cf.\ \cite[Vol.\hskip1.9ptII, p.\ 311]{kobayashi-nomizu}. On 
the other hand, $\,[\hh\ctr\hs]=0\,$ in $\,H^2(\bs,\bbK)\,$ when 
$\,\bbK=\bbR\hs$, since, choosing a connection $\,\nabla\hn\,$ compatible with a 
Riemannian fibre metric in $\,\bv$, we get $\,\ctr=0$.
\end{rem}

\section{Projectively flat connections}\label{pfco}
Let $\,\nabla\hn\,$ be a connection in a real\hs/com\-plex vector bundle 
$\,\bv\,$ over a real manifold $\,\bs$. Following Li, Yau and Zheng 
\cite{li-yau-zheng}, we call $\,\nabla\hn\,$ {\it projectively flat\/} if its 
curvature tensor $\,R\,$ equals $\,\rho\hskip1.pt\otimes\text{\rm Id}\,$ for 
some $\,2$-form $\,\rho\,$ on $\,\bs$, in the sense that 
$\,R\hh(\nh u,v)\psi=\rho\hs(\nh u,v)\hh\psi$ for all sections $\,\psi\,$ of 
$\,\bv\,$ and vector fields $\,u,v\,$ tangent to $\,M$. See also 
Section~\ref{nthm}.

This meaning of projective flatness is quite different from what the term 
traditionally refers to in the case of connections in the tangent bundle 
\cite[p.\ 915]{simon-00}.
\begin{rem}\label{prjfl}For a projectively flat connection, a $\,2$-form 
$\,\rho\,$ with $\,R=\rho\hh\otimes\text{\rm Id}\,$ is a constant 
multiple of the form $\,\ctr\,$ given by (\ref{ope}.ii). Thus, $\,\rho\,$ is 
closed, and, in the case of a real vector bundle, $\,[\hh\rho\hs]=0\,$ in 
$\,H^2(\bs,\bbR)\,$ according to Remark~\ref{cohom}.
\end{rem}
In the following lemma, the notation 
$\,\mathrm{D}\hs=\hs\nabla\hs+\,\hh\xi\otimes\mathrm{Id}\,$ means that 
$\,\mathrm{D}_v\psi=\nabla_{\hskip-2.2ptv}\psi+\hh\xi(v)\hs\psi\,$ whenever 
$\,v\,$ is a vector field tangent to $\,\bs\,$ and $\,\psi\,$ is a section 
of $\,\bv$. By (\ref{cur}) and (\ref{dxi}.c), the curvature tensors of such 
connections are related by
\begin{equation}\label{rdr}
R\hs^{\mathrm{D}}\nh(\nh u,v)\hs\psi\,\,
=\,\,R^\nabla\hn(\nh u,v)\hs\psi\,\,-\,\,[\hs(d\hs\xi)(\nh u,v)]\hs\psi\hs.
\end{equation}
\begin{lem}\label{prffl}For a real\hs/\nh com\-plex vector bundle\/ $\,\bv\,$ 
over a manifold, the assignment
\begin{equation}\label{dtn}
(\nabla,\hs\xi)\,\mapsto\,(\mathrm{D}\hh,\hs\xi)\hs,\hskip9pt
\mathrm{where}\hskip6pt\mathrm{D}\hs=\hs\nabla\hs+\,\hh\xi\otimes\mathrm{Id}
\end{equation}
defines a bijective correspondence between the set of all pairs\/ 
$\,(\nabla,\hs\xi)\,$ in which\/ $\,\nabla$ is a projectively flat 
connection in\/ $\,\bv$ and\/ $\,\xi\nh\,$ is a $\,1$-form on the base 
manifold such that\/ the curvature tensor of\/ $\,\nabla\nh,$ equals 
$\,\rho\hh\otimes\text{\rm Id}\,$ for $\,\rho=d\hs\xi$, and the set of all 
pairs\/ $\,(\mathrm{D}\hh,\hs\xi)\,$ consisting of a flat connection\/ 
$\,\mathrm{D}\nh\,$ in\/ $\,\bv\nh\,$ and a $\,1$-form\/ $\,\xi\nh\,$ on the 
base.
\end{lem}
This is obvious from (\ref{rdr}). Using Lemma~\ref{prffl} and 
Remark~\ref{prjfl} we now obtain the following conclusion.
\begin{cor}\label{ortbl}A real vector bundle over a manifold admits a 
projectively flat connection if and only if it admits a flat connection.
\end{cor}

\section{Skew-sym\-me\-try of the Ric\-ci tensor}\label{ssrt}
\begin{lem}\label{rctoc}A connection\/ $\,\nabla\hn\,$ in the tangent bundle\/ 
$\,\tb\,$ of a real surface\/ $\,\bs\,$ is projectively flat in the sense 
of Section\/~{\rm\ref{pfco}} if and only if\/ $\,\nabla\hn\,$ has 
skew-sym\-met\-ric Ric\-ci tensor, and then 
$\,R=\rho\hh\otimes\text{\rm Id}\,$ for the Ric\-ci tensor\/ $\,\rho\,$ 
of\/ $\,\nabla$.
\end{lem}
\begin{proof}If $\,R=\rho\hh\otimes\text{\rm Id}\,$ for {\it some\/} 
$\,2$-form $\,\rho$, then the Ric\-ci tensor of $\,\nabla\hn\,$ is 
skew-sym\-met\-ric, since it equals $\,\rho$. Conversely, let the Ric\-ci 
tensor $\,\rho\,$ be skew-sym\-met\-ric. As the discussion is local and 
$\,\dim\bs=2$, we may assume that $\,\bs\,$ is orientable and choose a 
$\,2$-form $\,\zeta\,$ on $\,\bs\,$ without zeros. Thus, 
$\,R=\zeta\otimes A\,$ for some bundle morphism $\,A:\tb\to\tb$. 
Skew-sym\-me\-try of $\,\rho\,$ now gives 
$\,0=\rho\hs(\nh u,u)=\zeta(\nh u,Au)$ for every vector field $\,u$, so that 
every nonzero vector tangent to $\,\hh\bs\,$ at any point $\,y$ is an 
eigenvector of $\,A_y$, and, consequently, $\,A_y$ is a multiple of 
$\,\mathrm{Id}\hh$, as required.
\end{proof}
We have the following obvious consequence of Lemmas~\ref{prffl} 
and~\ref{rctoc}.
\begin{cor}\label{dxtna}Given a surface\/ $\,\bs$, the assignment\/ 
$\,(\nabla,\hs\xi)\,\mapsto\,(\mathrm{D}\hh,\hs\xi)$, where\/ 
$\,\mathrm{D}\hs=\hs\nabla\hs+\,\hh\xi\otimes\mathrm{Id}\hh$, defines 
a bijective correspondence between the set of all pairs\/ 
$\,\hs(\nabla,\hs\xi)$ consisting of a connection\/ $\,\nabla\hn\,$ on\/ 
$\,\bs\,$ along with a $\,1$-form\/ $\,\xi\nh\,$ such that\/ $\,\hs d\hs\xi\,$ 
equals the Ric\-ci tensor of\/ $\,\nabla\nnh$, and the set of all pairs\/ 
$\,(\mathrm{D}\hh,\hs\xi)\,$ consisting of any flat connection\/ 
$\,\mathrm{D}\,$ on\/ $\,\bs\nh\,$ and any $\,1$-form\/ $\,\xi\nh\,$ on\/ 
$\,\bs$.
\end{cor}
\begin{rem}\label{torfm}In general, if connections $\,\nabla\hn\,$ and  
$\,\mathrm{D}\,$ on a surface are related by 
$\,\mathrm{D}\hs=\hs\nabla\hs+\,\hh\xi\otimes\mathrm{Id}\hh$, with a 
$\,1$-form $\,\xi$, then, obviously, $\,\tau\nh=\hs\t+\xi$, for the torsion 
$\,1$-forms $\,\t\,$ of $\,\nabla\hn\,$ and $\,\tau\hn\,$ of 
$\,\mathrm{D}$, defined as in Section~\ref{prel}.
\end{rem}

\section{The case of closed surfaces}\label{tccs}
The next result generalizes a theorem of Bla\-\v zi\'c and Bo\-kan 
\cite{blazic-bokan}, mentioned in the Introduction.
\begin{thm}\label{clsrf}A closed surface admitting a connection with 
skew-sym\-met\-ric Ric\-ci tensor\/ $\,\nh\rho\hs\,$ is dif\-feo\-mor\-phic 
to\/ $\,\hs T\hh^2\nnh$ or the Klein bottle, and the $\,2$-form\/ $\,\rho\,$ 
is exact.
\end{thm}
\begin{proof}Exactness of $\,\rho\,$ is a consequence of Remark~\ref{prjfl}. 
Thus, in view of Lemma~\ref{rctoc} and Corollary~\ref{ortbl}, $\,\bs\,$ admits 
a flat connection. Our assertion is now immediate from a result of Milnor 
\cite{milnor}.
\end{proof}
Note that, being exact, $\,\rho\,$ in Theorem~\ref{clsrf} must vanish 
somewhere: if it did not, it would distinguish an orientation of the surface, 
for which the oriented integral of the positive form $\,\rho\,$ would be 
positive, thus contradicting the exactness of $\,\rho\,$ via the Stokes 
theorem.

Bla\-\v zi\'c and Bo\-kan \cite{blazic-bokan} exhibited a non\hh-flat 
tor\-sion\-free connection $\,\nabla\hn\,$ with skew-sym\-met\-ric Ric\-ci 
tensor on the torus $\,T\hh^2\nnh$, which admits a $\,\nabla\nh$-par\-al\-lel 
almost complex structure, and belongs to a family constructed by Si\-mon 
\cite[p.\ 322]{simon-87}.

Theorem~\ref{clsrf} and Corollary~\ref{dxtna} yield the following description 
of all connections with skew-sym\-met\-ric Ric\-ci tensor on closed surfaces:
\begin{thm}\label{class}Let\/ $\,\bs\,$ be dif\-feo\-mor\-phic to\/ $\,T\hh^2$ 
or the Klein bottle. If\/ $\,\mathrm{D}\hn\,$ is any flat connection on\/ 
$\,\bs$, and\/ $\,\xi\nnh\,$ is a $\,1$-form on\/ $\,\bs$, then the 
connection\/ $\,\,\nabla=\,\mathrm{D}\hs-\,\hh\xi\otimes\mathrm{Id}$ on\/ 
$\,\bs\,$ has skew-sym\-met\-ric Ric\-ci tensor. Conversely, every connection 
with skew-sym\-met\-ric Ric\-ci tensor on\/ $\,\bs\,$ equals\/ 
$\,\mathrm{D}\hs-\,\hh\xi\otimes\mathrm{Id}\,$ for some such\/ 
$\,\mathrm{D}\hs$ and\/ $\,\xi$.
\end{thm}

\section{Wong's theorem}\label{wong}
Corollary~\ref{dxtna} leads to an obvious coordinate formula for connections 
$\,\nabla\hn\,$ with skew-sym\-met\-ric Ric\-ci tensor on surfaces, which 
produces all lo\-cal-e\-quiv\-a\-lence classes of such $\,\nabla$. 
Specifically, one needs to provide a flat connection $\,\mathrm{D}\hn\,$ along 
with a $\,1$-form $\,\xi$, and then set 
$\,\nabla\hs=\hskip2pt\mathrm{D}\hskip2pt-\,\hh\xi\otimes\mathrm{Id}\hh$. In a 
fixed coordinate system, $\,\mathrm{D}\,$ can be introduced by prescribing a 
basis of $\,u,v\,$ of $\,\mathrm{D}\hs$-par\-al\-lel vector fields, that is, 
four arbitrary functions subject just to one determinant condition (linear 
independence of $\,u\,$ and $\,v$), while $\,\xi\,$ amounts to two more 
arbitrary functions.

Using as many as six arbitrary functions is redundant, and their number is 
easily reduced. For instance, requiring $\,u\,$ to be the first coordinate 
vector field $\,\hs\partial_1$ leaves us with just four arbitrary functions 
(the first two of the six now being the constants $\,1\,$ and $\,0$). Another 
way of replacing six arbitrary functions with four consists in choosing 
$\,u\,$ as well as $\,v\,$ to be a product of a positive function and a 
coordinate vector field. In fact, whenever vector fields $\,u,v\,$ on a 
surface are linearly independent at each point, there exist, locally, 
coordinates $\,y^1\nnh,y^2$ with $\,u=e^\chi\hs\partial_1$, 
$\,v=e^\beta\hs\partial_2$ for some functions $\,\beta,\chi$. Namely, the 
distributions spanned by $\,u\,$ and $\,v$, being one\hh\diml, are integrable, 
and so their leaves are, locally, the level curves of some functions 
$\,y^1\nnh,y^2$ without critical points, which means that $\,u,v\,$ are 
functional multiples of $\,\hs\partial_1$ and $\,\hs\partial_2$, while 
positivity of the factor functions $\,e^\chi$ and $\,e^\beta$ is achieved by 
adjusting the signs of $\,y^1$ and $\,y^2\nnh$, if necessary.

It is this last approach that allows us to simplify Wong's result 
\cite[Theorem~4.2]{wong}, by dropping the assumption that $\,\rho\ne0$, and 
reducing the number of separate coordinate expressions from three to one:
\begin{thm}\label{wongs}For a tor\-sion\-free connection\/ $\,\nabla\nh\,$ on 
a surface\/ $\,\bs$, the Ric\-ci tensor\/ $\,\rho\,$  of\/ $\,\nabla\hn\,$ is 
skew-sym\-met\-ric if and only if every point of\/ $\,\bs\,$ has a 
neighborhood\/ $\,\,U$ with coordinates $\,y^1\nnh,y^2$ in which the 
component functions of\/ $\,\nabla\hn\,$ are 
$\,\vg_{\!11}^1=-\hs\partial_1\varphi$, 
$\,\vg_{\!22}^2=\hs\partial_2\varphi\,$ for some function $\,\varphi$, and\/ 
%$\,\vg_{\!11}^2=\vg_{\!22}^1=\vg_{\!12}^1=\vg_{\!12}^2$
$\,\vg_{\hskip-2.7ptjk}^{\hs l}=0\,$ unless $\,j=k=l$.
\end{thm}
\begin{proof}If $\,\rho\,$ is skew-sym\-met\-ric, we may choose, locally, a 
$\,1$-form $\,\xi\,$ with $\,\hs d\hs\xi=\rho\,$ and linearly independent 
vector fields $\,u,v\,$ that are $\,\mathrm{D}\hs$-par\-al\-lel, for the flat 
connection $\,\mathrm{D}\hs=\hs\nabla\hs+\,\hh\xi\otimes\mathrm{Id}\,$ (see 
Corollary~\ref{dxtna}), and then pick local coordinates $\,y^1\nnh,y^2$ such 
that $\,u=e^\chi\hh\partial_1$, $\,v=e^\beta\hh\partial_2$ for some functions 
$\,\beta,\chi$, as described above. Since $\,\nabla\hn\,$ is tor\-sion\-free, 
$\,\xi\,$ is the torsion $\,1$-form of $\,\mathrm{D}$, cf.\ 
Remark~\ref{torfm}. Thus, $\,\T=\hh\xi\wedge\mathrm{Id}\,$ is the torsion 
tensor of $\,\mathrm{D}$, and so $\,\xi(v)\hh u-\xi(u)\hh v=\T(v,u)=[u,v]$, 
while $\,[u,v]=[e^\chi\hh\partial_1,\hh e^\beta\hh\partial_2]$, so that 
the functions $\,\xi_j=\xi(\partial_j)$ are given by 
$\,\xi_1=-\hs\partial_1\beta$, $\,\xi_2=-\hs\partial_2\chi$. As 
$\,\mathrm{D}u=\mathrm{D}v=0\,$ and 
$\,\nabla\hs=\hskip2pt\mathrm{D}\hskip2pt-\,\hh\xi\otimes\mathrm{Id}$, we get 
$\,\nabla u=-\hs\xi\otimes u$, $\,\nabla v=-\hs\xi\otimes v$, which, for 
$\,\hs\partial_1=e^{-\chi}u$, $\,\hs\partial_2=e^{-\beta}v$, yields 
$\,\nabla\partial_1=-\hs(\xi+\hs d\chi)\otimes\partial_1$, 
$\,\nabla\partial_2=-\hs(\xi+\hs d\beta)\otimes\partial_2$. Since 
$\,\xi_1=-\hs\partial_1\beta\,$ and $\,\xi_2=-\hs\partial_2\chi$, setting 
$\,\varphi=\chi-\beta$, we obtain the required expressions for 
$\,\vg_{\hskip-2.7ptjk}^{\hs l}$.

Conversely, for a connection $\,\nabla\hn\,$ with 
$\,\vg_{\hskip-2.7ptjk}^{\hs l}$ as in the statement of the theorem, using 
(\ref{cur}) with $\,u,v\,$ replaced by $\,\hs\partial_1,\hs\partial_2$, 
and $\,\psi=\hs\partial_1$ or $\,\psi=\hs\partial_2$, we see that 
$\,R=\rho\hh\otimes\text{\rm Id}\,$ with 
$\,\rho_{12}=-\rho_{21}=-\hs\partial_1\hs\partial_2\varphi$, which completes 
the proof.
\end{proof}

\section{Left-in\-var\-i\-ant connections on Lie groups}\label{lclg}
We say that a connection $\,\nabla\hn\,$ on a manifold $\,\bs\,$ is {\it 
locally equivalent\/} to a connection $\,\nabla'$ on a manifold $\,\bs\hs'$ if 
every point of $\,\bs\,$ has a connected neighborhood $\,\,U\,$ with an 
af\-fine dif\-feo\-mor\-phism $\,\,U\nh\to U'$ onto an open subset $\,\,U'$ of 
$\,\bs\hs'\nnh$.

Here and in the next two sections we describe all 
lo\-cal-\hh e\-quiv\-a\-lence types of locally homogeneous connections with 
skew-sym\-met\-ric Ric\-ci tensor on surfaces. They all turn out to be 
represented by left-in\-var\-i\-ant connections on Lie groups, which is why we 
discuss the Lie\hs-group case first.
\begin{ex}\label{hlfpl}Given an area form 
$\,\alpha\in[\plane^*]^{\wedge2}\nh\smallsetminus\{0\}\,$ in a two\hh\diml\ 
real vector space $\,\plane\,$ and a one\hh\diml\ vector sub\-space 
$\,\line\,$ of $\,\plane$, we denote by $\,\mathrm{D}\,$ the standard 
(trans\-la\-tion-in\-var\-i\-ant) flat tor\-sion\-free connection on 
$\,\plane$, by $\,\xi\,$ be the $\,1$-form on $\,\plane\,$ given by 
$\,\xi_y(v)=\alpha(y,v)$, for $\,y\in\plane\,$ and $\,v\in\plane=T_y\plane$, 
and by $\,\bs\,$ a fixed side of $\,\line\,$ in $\,\plane\,$ (that is, a 
connected component of $\,\plane\smallsetminus\line$). Let $\,G\,$ be the 
two\hh\diml\ non\hs-Abel\-i\-an connected subgroup of the group 
$\,\mathrm{SL}\hs(\plane)$, formed by those elements of 
$\,\mathrm{SL}\hs(\plane)\,$ which leave $\,\line\,$ invariant and operate in 
$\,\line$ via multiplication by positive scalars. Since $\,G\,$ acts on 
$\,\bs\,$ freely and transitively, choosing a point in $\,\bs\,$ we identify 
$\,\bs\,$ with $\,G\,$ and treat the action as consisting of the left 
translations in $\,G$. The restrictions of $\,\mathrm{D}\,$ and $\,\xi\,$ to 
$\,\bs=G\,$ then are invariant under all left translations, and hence so is 
the connection 
$\,\nabla\nh=\hskip1pt\mathrm{D}\hskip1pt-\,\hh\xi\otimes\mathrm{Id}$ on 
$\,\bs=G$. As $\,d\hs\xi=2\hh\alpha$, where $\,\alpha\,$ is now treated as a 
constant (trans\-la\-tion-in\-var\-i\-ant) $\,2$-form on $\,\plane$, 
Corollary~\ref{dxtna} implies that $\,\nabla\hn\,$ has the skew-sym\-met\-ric 
Ric\-ci tensor $\,\rho=2\hh\alpha\ne0\,$ and, by Remark~\ref{torfm}, the 
torsion $\,1$-form of $\,\nabla\hn\,$ is $\,\t=-\,\hh\xi\ne0$.
\end{ex}
We always identify the elements of the Lie algebra $\,\mathfrak{g}\,$ of any 
Lie group $\,G$ with left-in\-var\-i\-ant vector fields on $\,G$. The symbol 
$\,\mathfrak{sl}\hh(\plane)\,$ denotes the Lie algebra of traceless 
endomorphisms of a real vector space $\,\plane$.
\begin{thm}\label{liegp}For any connected Lie group\/ $\,G$, the 
left-in\-var\-i\-ant connections $\,\nabla\hn\,$ on\/ $\,G\,$ which are 
projectively flat in the sense of Section\/~{\rm\ref{pfco}} are in a bijective 
correspondence with pairs $\,(\Psi,f\hs)\hs$ formed by any Lie-al\-ge\-bra 
homomorphism\/ $\,\Psi:\mathfrak{g}\to\mathfrak{sl}\hh(\mathfrak{g})\,$ and 
any linear functional $\,f\in\mathfrak{g}^*\nnh$. For $\,u,v\in\mathfrak{g}$, 
this correspondence is given by 
$\,\nabla_{\hskip-2.2ptu}v=(\Psi u)\hh v+f(u)\hh v$, and\/ $\,\nabla\hn\,$ has 
the Ric\-ci tensor $\,\rho\,$ with $\,\rho\hs(\nh u,v)=f([u,v])$.
\end{thm}
\begin{proof}A left-in\-var\-i\-ant connection $\,\nabla\hn\,$ on $\,G\,$ 
clearly amounts to a linear operator 
$\,\mathfrak{g}\ni u\mapsto\nabla_{\hskip-2.2ptu}$ valued in linear 
endomorphisms of $\,\mathfrak{g}$. Decomposing $\,\nabla_{\hskip-2.2ptu}$ into 
a traceless part and a multiple of $\,\mathrm{Id}\hh$, we obtain the formula 
$\,\nabla_{\hskip-2.2ptu}v=(\Psi u)\hh v+f(u)\hh v\,$ describing a 
bijective correspondence between left-in\-var\-i\-ant connections 
$\,\nabla\hn\,$ on $\,G\,$ and pairs $\,(\Psi,f\hs)$, in which 
$\,\Psi:\mathfrak{g}\to\mathfrak{sl}\hh(\mathfrak{g})\,$ is a linear operator, 
and $\,f\in\mathfrak{g}^*\nnh$. By (\ref{cur}), the curvature tensor of 
$\,\nabla\hn\,$ is given by 
$\,R\hh(\nh u,v)=\nabla_{\nh[u,v]}+\nabla_{\hskip-2.2ptv}\nabla_{\hskip-2.2ptu}
-\nabla_{\hskip-2.2ptu}\nabla_{\hskip-2.2ptv}$, cf.\ (\ref{ope}.i), for 
$\,u,v,w\in\mathfrak{g}$, that is, 
$\,R\hh(\nh u,v)=f([u,v])\hs\mathrm{Id}+\Psi[u,v]-[\Psi u,\Psi v]$, where the 
first two occurrences of $\,[\hskip2.5pt,\hskip1pt]\,$ represent the 
Lie\hs-al\-ge\-bra operation in $\,\mathfrak{g}$, and the last one stands for 
the commutator in $\,\hs\mathfrak{sl}\hh(\mathfrak{g})$. On the other hand, 
projective flatness of $\,\hs\nabla$ means that 
$\,R\hh(\nh u,v)=\rho\hs(\nh u,v)\hs\text{\rm Id}\,$ for some $\,2$-form 
$\,\rho\,$ (which must then coincide with the Ric\-ci tensor of $\,\nabla$). 
Equating the traceless parts of the last two expressions for 
$\,R\hh(\nh u,v)$, we see that $\,\nabla\hn\,$ is projectively flat if and only 
if $\,\Psi[u,v]=[\Psi u,\Psi v]\,$ for all $\,u,v\in\mathfrak{g}$, and then 
$\,\rho\hs(\nh u,v)=f([u,v])$. This completes the proof.
\end{proof}
Theorem~\ref{liegp} leads to an explicit description of all 
lo\-cal-e\-quiv\-a\-lence types of left-in\-var\-i\-ant connections 
$\,\nabla\hn\,$ with skew-sym\-met\-ric Ric\-ci tensor on two\hh\diml\ Lie 
groups $\,G$. It is convenient to distinguish three cases, based on the rank 
(dimension of the image) of the Lie\hs-al\-ge\-bra homomorphism $\,\Psi\,$ 
associated with $\,\nabla$, which assumes the values $\,0,1\,$ and $\,2$.
Note that the Lie algebras $\,\mathfrak{g}\,$ of the groups $\,G\,$ in 
question represent just two isomorphism types (Abelian and non\hs-Abel\-i\-an).

First, connections $\,\nabla\hn\,$ as above with 
$\,\mathrm{rank}\hskip3pt\Psi=0\,$ (that is, $\,\Psi=0$) are, by 
Theorem~\ref{liegp}, in a one-to-one correspondence with linear functionals 
$\,f\in\mathfrak{g}^*\nnh$.

Secondly, those of our connections having $\,\mathrm{rank}\hskip3pt\Psi=1\,$ 
are precisely the connections $\,\nabla\hn\,$ of the form 
$\,\nabla_{\hskip-2.2ptu}v=q(u)\hh Bv+f(u)\hh v$, for all 
$\,u,v\in\mathfrak{g}$, with any fixed 
$\,B\in\mathfrak{sl}\hh(\mathfrak{g})\smallsetminus\{0\}\,$ and 
$\,q,f\in\mathfrak{g}^*$ such that $\,q\ne 0\,$ and 
$\,\mathrm{Ker}\hh\,q\,$ contains the commutant ideal 
$\,[\mathfrak{g},\mathfrak{g}]$. (Note that 
$\,[\mathfrak{g},\mathfrak{g}]=\{0\}\,$ if $\,\mathfrak{g}\,$ is Abelian, and 
$\,\hs\dim\hskip2pt[\mathfrak{g},\mathfrak{g}]=1$ if it is not, while, for 
$\,q\,$ and $\,B\,$ which are both nonzero, $\,u\mapsto q(u)\hh B\,$ 
\hbox{is a Lie\hs-} al\-ge\-bra homomorphism 
$\,\mathfrak{g}\to\mathfrak{sl}\hh(\mathfrak{g})\,$ if and only if 
$\,[\mathfrak{g},\mathfrak{g}]\subset\mathrm{Ker}\hh\,q$.) 

Finally, the case $\,\mathrm{rank}\hskip3pt\Psi=2\,$ occurs only for 
non\hs-Abel\-i\-an $\,\mathfrak{g}\,$ %(Remark~\ref{sltwo} below). 
(by Theorem~\ref{lsasl}(i) in Appendix A).
Our connections 
then have the form 
$\,\nabla_{\hskip-2.2ptu}v=(\Psi u)\hh v+f(u)\hh v$, for 
$\,u,v\in\mathfrak{g}$, with $\,\Psi\,$ explicitly described as follows: 
$\,\Psi u=A\,$ and $\,\Psi v=B$, where $\,u,v\,$ is a fixed basis of 
$\,\mathfrak{g}\,$ with $\,[u,v]=u\,$ and 
$\,A,B\in\mathfrak{sl}\hh(\mathfrak{g})\,$ are given by
\begin{equation}\label{aww}
Aw\,=\,w\hh',\hskip13ptAw\hh'\hh=\,0\hs,\hskip13ptBw\,=\,w\nh/2\hs,\hskip13pt
Bw\hh'\hh=\,-\hs w\hh'\hskip-2.2pt/2\hs,
\end{equation}
in an arbitrary basis $\,w,w\hh'$ of $\,\mathfrak{g}$. See 
Theorem~\ref{lsasl}(ii) in Appendix A.

\section{Flat locally homogeneous connections}\label{flhc}
Let the Ric\-ci tensor $\,\rho\,$ of a connection $\,\nabla\hn\,$ on a surface 
$\,\bs\,$ be skew-sym\-met\-ric. In the open set $\,\,U\nh\,$ where 
$\,\rho\ne0$, the determinant bundle $\,[\tb]^{\wedge2}$ is trivialized by 
$\,\rho$, and so $\,\rho\,$ restricted to $\,\,U\,$ is {\it re\-cur\-rent\/} 
in the sense that $\,\nabla\nnh\rho=\phi\otimes\rho\,$ for some $\,1$-form 
$\,\phi\,$ defined just on $\,\,U\nh$. If $\,\,U\hn\,$ is nonempty, we call 
$\,\phi\,$ the {\it re\-cur\-rence form\/} of $\,\rho$. On $\,\,U\,$ one then 
has
\begin{equation}\label{rcr}
d\hs\phi\,=\,2\rho\hs.%\,+\,\phi\wedge\t\hs.
\end{equation}
In fact, the lo\-cal-co\-or\-di\-nate form 
$\,\rho_{jk,l}=\phi_{\hh l}\rho_{jk}$ of the re\-cur\-rence relation, combined 
with the Ric\-ci identity, gives $\,(\phi_{\hh l,m}-\phi_{m,l})\rho_{jk}
=\rho_{jk,lm}-\rho_{jk,ml}
=R_{mlj}{}^s\rho_{sk}+R_{mlk}{}^s\rho_{js}+\T_{\hs\!lm}^{\hh s}\rho_{jk,s}
=(2\rho_{ml}+\t_l\phi_m-\t_m\phi_{\hh l})\rho_{jk}$ (where $\,\t\,$ is the 
torsion $\,1$-form of $\,\nabla$), the last equality being immediate as 
$\,R_{mlj}{}^s=\rho_{ml}\delta_j^s$ in view of Lemma~\ref{rctoc}, and 
$\,\T_{\hs\!lm}^{\hh s}=\t_l\delta_m^s-\t_m\delta_l^s$ (cf.\ 
Section~\ref{prel}). Cancelling the factor $\,\rho_{jk}$ and noting that 
$\,\phi_{\hh l,m}-\phi_{m,l}=(d\hs\phi)_{ml}+\t_l\phi_m-\t_m\phi_{\hh l}$ (by 
(\ref{dxt}) and (\ref{dxi}.a)), we obtain (\ref{rcr}).

Garc\'\i a\hh-R\'\i o, Kupeli and V\'azquez\hh-Lorenzo 
\cite[p.\ 144]{garcia-rio-kupeli-vazquez-lorenzo} showed that a 
tor\-sion\-free connection with skew-sym\-met\-ric Ric\-ci tensor on a surface 
can be locally symmetric only if it is flat. By (\ref{rcr}), this remains true 
for connections with torsion.
\begin{lem}\label{linvt}For a connection $\,\nabla\hn\,$ on an $\,n$\diml\ 
manifold $\,\bs$, if%any of the following four conditions 
%the existence of any of the following ... objects implies that 
\begin{enumerate}
  \def\theenumi{{\rm\alph{enumi}}}
\item[(i)] $\nabla_{\hskip-2pte_j}e_k=\vg_{\hskip-2.7ptjk}^{\hs l}e_{\hh l}$ and 
$\,\T(e_j,e_k)=\T_{\!jk}^{\hh l}e_{\hh l}$ for some vector fields $\,e_1,\dots,e_n$ 
trivializing $\,\tb\,$ and some constants 
$\,\vg_{\hskip-2.7ptjk}^{\hs l},\T_{\!jk}^{\hh l}$, where\/ 
$\,j,k,l\in\{1,\dots,n\}$, repeated indices are summed over, and\/ $\,\T\,$ is 
the torsion tensor of $\,\nabla\nnh$, or
\item[(ii)] $\nabla\,$ is flat and has parallel torsion, or
\item[(iii)] $n=2\,$ and\/ $\,\nabla\hn\,$ is a part of a locally homogeneous 
triple\/ $\,(\nabla\hn,\hs\xi,g)\,$ also including a nonzero $\,1$-form\/ 
$\,\xi\,$ and a pseu\-\hbox{do\hs-}Riem\-ann\-i\-an metric\/ $\,g\,$ on\/ 
$\,\bs$,
\end{enumerate}
then\/ $\,\nabla\hn\,$ is locally equivalent to a left-in\-var\-i\-ant connection 
on some Lie group.
\end{lem}
\begin{proof}As 
$\,[u,v]=\nabla_{\hskip-2.2ptu}v-\nabla_{\hskip-2.2ptv}u-\T(\nh u,v)\,$ for 
vector fields $\,u,v$, our $\,e_j$ in (i) span a Lie algebra 
$\,\mathfrak{h}\,$ of vector fields, trivializing $\,\tb$, and so (i) follows 
from Theorem~\ref{latri} in Appendix B. Next, if $\,\nabla\nh\,$ is flat and 
$\,\nabla\T=0$, choosing, locally, $\,\nabla\nh$-par\-al\-lel vector fields 
$\,e_1,\dots,e_n$ trivializing $\,\tb$, we see that the assumptions in (i) 
hold with $\,\vg_{\hskip-2.7ptjk}^{\hs l}=0$, and so (i) implies (ii). 
Finally, let $\,(\nabla\hn,\hs\xi,g)\,$ be as in (iii), with $\,n=2$. We 
denote by $\,u\,$ the unique vector field with $\,g(\nh u,\,\cdot\,)=\xi$. A 
second vector field $\,v\,$ is defined by $\,g(v,v)=0\,$ and 
$\,g(\nh u,v)=1\,$ (if $\,g(\nh u,u)=0$), or  $\,|\hs g(v,v)|=1\,$ and 
$\,g(\nh u,v)=0\,$ (if $\,g(\nh u,u)\ne0$); in the latter case, $\,v\,$ is 
determined only up to a sign. In both cases, for reasons of naturality, the 
triple $\,(\nabla\hn,u,\pm v)\,$ is locally homogeneous, and so, since, 
locally, $\,u\,$ and $\,\pm v\,$ trivialize $\,\tb$, the covariant derivatives 
$\,\nabla_{\hskip-2.2ptu}u,\hs\pm\nabla_{\hskip-2.2ptu}v,
\hs\pm\nabla_{\hskip-2.2ptv}u,\hs\nabla_{\hskip-2.2ptv}v$, as well as 
$\,\T(\nh u,v)$, are linear combinations of $\,u\,$ and $\,\pm v\,$ with 
constant coefficients. Fixing, locally, the sign $\,\pm\hs$, we obtain (iii) 
from (i) for $\,e_1=u\,$ and $\,e_2=\pm v$.
\end{proof}
\begin{rem}\label{lhlie}Every locally homogeneous tor\-sion\-free connection 
on a surface is locally equivalent either to the Le\-vi-Ci\-vi\-ta connection 
of the standard sphere, or to a left-in\-var\-i\-ant connection on a Lie 
group. This is obvious from O\-poz\-da's local classification of such 
connections \cite[Theorem 1.1]{opozda}: for $\,u,v,U,V\,$ as in 
\cite[formula (1.4)]{opozda}, we may apply Lemma~\ref{linvt}(i) to 
$\,e_1=u\hh U\,$ and $\,e_2=vV$, while left-in\-var\-i\-ant 
pseu\-\hbox{do\hs-}Riem\-ann\-i\-an metrics on a two\hh\diml\ 
non\hs-Abel\-i\-an Lie group realize all non\hh-flat con\-stant-cur\-va\-ture 
metric types other than the standard sphere.
\end{rem}
\begin{ex}\label{yesli}Let the tangent bundle $\,\tb\,$ of a simply connected 
surface $\,\bs\,$ be trivialized by vector fields $\,u,v\,$ such that 
$\,[u,v]=2(u-v)$. The $\,1$-form $\,\xi\,$ on $\,\bs\,$ with 
$\,\xi(u)=\xi(v)=4\,$ is closed, as $\,(d\hs\xi)(\nh u,v)=0\,$ by 
(\ref{dxi}.c). For any fixed function $\,\varphi:\bs\to\bbR\,$ with 
$\,\xi=d\hs\varphi$, the connection $\,\nabla\hn\,$ on $\,\bs\,$ with
\begin{equation}\label{cnc}
\nabla_{\hskip-2.2ptu}u=u\hs,\hskip12pt\nabla_{\hskip-2.2ptu}v
=-\hs v\hs,\hskip12pt\nabla_{\hskip-2.2ptv}u=u+v\hs,\hskip12pt
\nabla_{\hskip-2.2ptv}v=e^{-\hh\varphi}u-v\hs,
\end{equation}
is locally equivalent to a left-in\-var\-i\-ant connection on a Lie group, 
although, as $\,\hs\varphi$ is nonconstant, this is not immediate from 
Lemma~\ref{linvt}(i) for $\,e_1=u\,$ and $\,e_2=v$.

In fact, the $\,1$-form $\,\eta\,$ on $\,\bs\,$ with $\,\eta(u)=0\,$ and 
$\,\eta(v)=e^{-\hh\varphi/2}$ is closed, as $\,d_u\varphi=d_v\varphi=4$, 
and so (\ref{dxi}.c) gives $\,(d\eta)(\nh u,v)=0$. Choosing a function 
$\,\gamma$ with $\,d\gamma=\eta\,$ and setting 
$\,\chi=e^{-\hh\varphi/2}\tanh\gamma$, we have $\,d_u\chi=-2\chi\,$ and 
$\,d_v\chi=\chi^2-2\chi-e^{-\hh\varphi}\nnh$. Now, for $\,w=v+\chi u$, the 
definition of $\,\nabla\hn\,$ yields $\,\nabla_{\hskip-2.2ptu}u=u$, 
$\,\nabla_{\hskip-2.7ptw}u=u+w\,$ and 
$\,\nabla_{\hskip-2.2ptu}w=\nabla_{\hskip-2.7ptw}w=-\hs w$. Our claim now 
follows from Lemma~\ref{linvt}(i) with $\,e_1=u\,$ and $\,e_2=w$, as 
$\,\T(\nh u,w)=\T(\nh u,v)=\nabla_{\hskip-2.2ptu}v
-\nabla_{\hskip-2.2ptv}u-[u,v]=-3u$.

We will not use the eas\-i\-ly-ver\-i\-fied fact that $\,\nabla\hn\,$ is flat.
\end{ex}
\begin{lem}\label{lieor}Let\/ $\,(\mathrm{D}\hh,\hs\xi)\,$ be a locally 
homogeneous pair consisting of a flat connection $\,\mathrm{D}\hn\,$ on a 
surface $\,\bs\hn\,$ and a $\,1$-form\/ $\,\xi\hn\,$ on\/ $\,\bs$. Suppose 
that\/ $\,\mathrm{D}\hn\,$ is tor\-sion\-free, or\/ $\,\xi\hn\,$ is the 
torsion $\,1$-form of\/ $\,\mathrm{D}$, cf.\ Section~\ref{prel}. Decomposing 
the $\,2$-ten\-sor\/ $\,\hs\mathrm{D}\hs\xi$ uniquely as\/ 
$\,\mathrm{D}\hs\xi=\om+g$, where\/ $\,\om\,$ is skew-sym\-met\-ric and\/ 
$\,g\,$ is symmetric, we then have\/ $\,d\hs\xi=2\hh\om$. Furthermore, unless
\begin{enumerate}
  \def\theenumi{{\rm\alph{enumi}}}
\item[(a)] $\mathrm{D}\,$ is locally equivalent to a left-in\-var\-i\-ant 
connection on a Lie group\/ $\,G\,$ in such a way that\/ $\,\xi\,$ 
corresponds to a left-in\-var\-i\-ant\/ $\,1$-form on\/ $\,G$,
\end{enumerate}
the following three conditions must be satisfied\/{\rm:}
\begin{enumerate}
  \def\theenumi{{\rm\alph{enumi}}}
\item[(b)] $\xi\,$ and\/ $\,\mathrm{D}\hs\xi\,$ are nonzero everywhere,
\item[(c)] $\mathrm{D}\hs\xi\nh=\hs\om-\hh c\hskip1.5pt\xi\otimes\xi$ for some 
constant\/ $\,c\hh$,
\item[(d)] the $\,2$-form\/ $\,\om=d\hs\xi/2\,$ is parallel and nonzero.
\end{enumerate}
\end{lem}
\begin{proof}By (\ref{dxt}), $\,d\hs\xi=2\hh\om$, as $\,\t\wedge\hs\xi=0$. 
Suppose that (a) does not hold. This gives (b): the case where 
$\,\mathrm{D}\,$ is tor\-sion\-free and $\,\mathrm{D}\hs\xi=0\,$ is 
excluded since it would imply (a) with an Abelian group $\,G$, while 
the case of parallel torsion would lead to (a) in view of 
Lemma~\ref{linvt}(ii). 

The rank of $\,g$, the symmetric part of $\,\mathrm{D}\hs\xi$, is 
constant on $\,\bs\,$ and equal to $\,0,1\,$ or $\,2$. If (a) fails, 
we must have $\,\mathrm{rank}\hskip3ptg\le1$, and so, locally, 
$\,g=\pm\hs\eta\otimes\eta\,$ for some $\,1$-form $\,\eta$. In fact, 
if $\,g\,$ were of rank $\,2$, Lemma~\ref{linvt}(iii) applied to 
the triple $\,(\mathrm{D}\hh,\hs\xi,g)\,$ would yield (a) (as 
$\,\xi\ne0\,$ by (b)). Furthermore, $\,\eta\,$ must be a constant 
multiple of $\,\xi$, for if it were not, Lemma~\ref{linvt}(iii) for 
the triple $\hs(\mathrm{D}\hh,\hs\xi,g+\xi\otimes\xi)\hs$ would imply (a) 
again. This gives (c).

Still assuming that (a) is not satisfied, we will now prove (d). 
Namely, if $\,\om=d\hs\xi/2\,$ were not parallel, $\,\om\,$ would be nonzero 
everywhere due to local homogeneity of $\,(\mathrm{D},\xi)$. Hence 
$\,\om\,$ would be re\-cur\-rent, in the sense that 
$\,\mathrm{D}\hh\om=\zeta\otimes\om\,$ for some nonzero 
$\,1$-form $\,\zeta\,$ (cf.\ the lines preceding (\ref{rcr})). As 
$\,\mathrm{D}\,$ is flat, so is the connection induced by $\,\mathrm{D}\,$ 
in the bundle $\,[\tab]^{\wedge2}\nnh$. Thus, locally, 
$\,e^{-\chi}\nnh\om\,$ is parallel for some function $\,\chi$, and so 
$\,\mathrm{D}\hh\om=d\chi\otimes\om$. It would now follow that 
$\,\zeta=d\chi\,$ and $\,d\hs\zeta=0$. However, since we assumed that 
$\,\om=d\hs\xi/2\,$ is {\it not\/} parallel, $\,\xi\,$ cannot be a 
constant multiple of $\,\zeta$, so that Lemma~\ref{linvt}(iii) 
applied to $\,(\mathrm{D}\hh,\hs\xi,\hs\xi\otimes\xi+\zeta\otimes\zeta)\,$ 
would give (a) as before. Thus, $\,\om=d\hs\xi/2\,$ is parallel.

To show that $\,\om\ne0$, suppose, on the contrary, that $\,\om=d\hs\xi/2\,$ 
vanishes identically (and (a) does not hold). Choosing, locally, a 
function $\,\varphi\,$ with $\,d\hs\varphi=\xi$, we can now rewrite the 
equality $\,\mathrm{D}\hs\xi=-c\hskip1.5pt\xi\otimes\xi\,$ (cf.\ (c)) as 
$\,\mathrm{D}\eta=0\,$ for $\,\eta=e^{c\hh\varphi}\xi$. Since 
$\,\mathrm{D}\,$ is flat, we may now select, locally, 
$\,\mathrm{D}\hs$-par\-al\-lel vector fields 
$\,u,v\,$ with $\,\eta(u)=1\,$ and $\,\eta(v)=0$. Setting 
$\,w=e^{c\hh\varphi}u\,$ we have 
$\,\mathrm{D}_vv=\mathrm{D}_wv=\mathrm{D}_vw=0\,$ and 
$\,\mathrm{D}_ww=c\hh w$, since $\,\mathrm{D}u=\mathrm{D}v=0$, while 
$\,d\hs\varphi=\xi=e^{-c\hh\varphi}\eta$, so that 
$\,d_u\varphi=e^{-c\hh\varphi}$ and $\,d_v\varphi=0$. Similarly, as the 
torsion tensor $\,\T\,$ of $\,\mathrm{D}\,$ equals $\,0\,$ (when 
$\,\mathrm{D}\,$ is tor\-sion\-free), or $\,\T=\xi\wedge\mathrm{Id}\,$ (when 
$\,\xi\,$ is the torsion 
$\,1$-form of $\,\mathrm{D}$), we get $\,\T(v,w)=0$ or, respectively, 
$\,\T(v,w)=-\xi(w)\hh v=-\hh\eta(u)\hh v=-\hs v$. Lemma~\ref{linvt}(i) with 
$\,n=2$, $\,e_1=v\,$ and $\,e_2=w\,$ now yields (a), and the resulting 
contradiction proves (d).
\end{proof}
\begin{lem}\label{flnli}Every flat locally homogeneous connection\/ 
$\,\nabla\hn\,$ on a surface $\,\bs\,$ is locally equivalent to a 
left-in\-var\-i\-ant connection on some Lie group.
\end{lem}
\begin{proof}Let $\,\t\,$ be the torsion form of $\,\nabla\hn\,$ (see 
Section~\ref{prel}). We assume that conditions (b) -- (d) in 
Lemma~\ref{lieor} are satisfied by $\,\mathrm{D}=\nabla\,$ and $\,\xi=\t$, 
since otherwise our assertion follows from Lemma~\ref{lieor}(a). Thus, 
$\,\nabla\t=\hs\om-\hh c\hskip1.5pt\t\otimes\t\,$ for the nonzero 
parallel $\,2$-form $\,\om=d\hs\t/2\,$ and a constant $\,c\hh$.

As $\,\nabla\om=0$, while $\,\om\ne0\,$ and $\,\t\ne0$, on some 
neighborhood $\,\,U\,$ of any given point of $\,\bs,$ we may choose 
$\,\nabla\nh$-par\-al\-lel vector fields $\,w,w\hh'$ such that 
$\,\hs\om(w,w\hh'\hh)=1$ and $\,\t(w)\hh\t(w\hh'\hh)\ne0\,$ 
everywhere in $\,\,U\nh$. Let us now define functions 
$\,P,Q,\varphi$, vector fields $\,u,v$, and $\,1$-forms 
$\,\zeta,\eta,\hs\xi\,$ on $\,\,U\,$ by $\,P=\t(w)$, $\,Q=\t(w\hh'\hh)$, 
$\,\varphi=2\hs\log\hs|PQ\hh|$, $\,u=Qw-Pw\hh'\nnh$, 
$\,v=3\hh(P^{-1}w+Q^{-1}w\hh'\hh)/2$, 
$\,\zeta=dP$, $\,\eta=d\hh Q$, and $\,\xi=d\hs\varphi$. We have
\begin{equation}\label{zwe}
3\hs\zeta(w)=P^2,\hskip8pt3\hs\zeta(w\hh'\hh)=PQ-3,\hskip8pt
3\hh\eta(w)=PQ+3,\hskip8pt3\hh\eta(w\hh'\hh)=Q^2.
\end{equation}
In fact, the relation $\,\nabla\t=\hs\om-\hh c\hskip1.5pt\t\otimes\t\,$ 
evaluated on the pairs 
$\,(w,w),\hs(w,w\hh'\hh),\hs(w\hh'\nnh,w)$ and $\,(w\hh'\nnh,w\hh'\hh)\,$ 
yields $\,\zeta(w)=-\hs c\hs P^2\nnh$, $\,\zeta(w\hh'\hh)=-\hs c\hs PQ-1$, 
$\,\eta(w)=-\hs c\hs PQ+1$, $\,\eta(w\hh'\hh)=-\hs c\hs Q^2\nnh$. As 
$\,d\hs\zeta=ddP=0$, (\ref{dxi}.c) gives 
$\,0=(d\hs\zeta)(w,w\hh'\hh)=-\hs(3\hh c+1)\hh P$, and hence 
$\,c=-1/3$, so that the preceding equalities become (\ref{zwe}). (As 
$\,\T=\t\wedge\mathrm{Id}\hh$, cf.\ Section~\ref{prel}, 
$\,[w,w\hh'\hh]=\T(w\hh'\nnh,w)=Qw-Pw\hh'\nnh$.) Combining (\ref{zwe}) with 
the relations $\,\xi=d\hs\varphi=2\hs(Q^{-1}\eta+P^{-1}\zeta)\,$ and 
$\,\nabla w=\nabla w\hh'\nh=0$, we get $\,\xi(u)=\xi(v)=4\,$ and 
(\ref{cnc}) (for our $\,\nabla,u,v,\hh\varphi$). Our assertion now follows 
from Example~\ref{yesli}.
\end{proof}

\section{The general locally homogeneous case}\label{glhc}
The following theorem, combined with the discussion, at the end of 
Section~\ref{lclg}, of certain left-in\-var\-i\-ant connections on Lie groups, 
completes our description of all locally homogeneous connections with 
skew-sym\-met\-ric Ric\-ci tensor on surfaces.
\begin{thm}\label{genrl}Every locally homogeneous connection\/ $\,\nabla\hn\,$ 
with skew-sym\-met\-ric Ric\-ci tensor on a surface $\,\bs\,$ is locally 
equivalent to a left-in\-var\-i\-ant connection on a Lie group.
\end{thm}
\begin{proof}If $\,\nabla\nnh\,$ is flat, we can use Lemma~\ref{flnli}. 
Suppose now that $\,\nabla\hn\,$ is not flat. As the 
Ric\-ci tensor $\,\rho\,$ is nonzero at every point (cf.\ 
Lemma~\ref{rctoc}), the re\-cur\-rence form $\,\phi\,$ with (\ref{rcr}) 
is defined and nonzero everywhere in $\,\bs$. (If $\,\phi\,$ 
were zero somewhere, it would vanish identically in view of local 
homogeneity of $\,\nabla$.) Let $\,\xi=\phi/2$. By (\ref{rcr}) and 
Corollary~\ref{dxtna}, the connection 
$\,\mathrm{D}\hs=\hs\nabla\hs+\,\hh\xi\otimes\mathrm{Id}\,$ is flat, as well 
as locally homogeneous (due to naturality of $\,\xi$), and has the 
torsion $\,1$-form $\,\t+\xi$, where $\,\hs\t$ is the torsion 
$\,1$-form of $\,\nabla\nnh$, cf.\ Remark~\ref{torfm}.

If $\,\xi\,$ and $\,\t\,$ are linearly independent at each point, our claim 
follows from Lemma~\ref{linvt}(iii) for the triple 
$\,(\mathrm{D}\hh,\hs\xi,g)\,$ with $\,g=\xi\otimes\xi+\t\otimes\t$. Let us 
now consider the remaining case where $\,\t=s\hh\xi\,$ for some 
$\,s\in\bbR\hs$, and $\,\mathrm{D}\,$ has the torsion $\,1$-form 
$\,\t+\xi=(s+1)\hs\xi$.

By Lemma~\ref{flnli}, $\,\mathrm{D}\,$ is locally equivalent to a 
left-in\-var\-i\-ant connection on some Lie group. If $\,s\ne-1$, our 
$\,\xi\,$ equals $\,(s+1)^{-1}$ times the torsion $\,1$-form of 
$\,\mathrm{D}$, and the original connection 
$\,\nabla\hn=\hskip1pt\mathrm{D}\hskip1pt-\,\hh\xi\otimes\mathrm{Id}\hh$, 
obtained from $\,\mathrm{D}\,$ via a natural formula, is also locally 
equivalent to a left-in\-var\-i\-ant connection on a Lie group, as required.

Suppose now that $\,s=-1$. Thus, the flat connection $\,\mathrm{D}\,$ is 
tor\-sion\-free, and the pair $\,(\mathrm{D}\hh,\hs\xi)\,$ is locally 
homogeneous, since so is $\,\nabla$. We may identify $\,\bs$, locally, with an 
open convex set $\,\,U\,$ in a real af\-fine plane $\,\plane$, in such a way 
that $\,\mathrm{D}\,$ equals, on $\,\,U\nh$, the standard 
trans\-la\-tion-in\-var\-i\-ant flat tor\-sion\-free connection of $\,\plane$.

As $\,\nabla\hn=\hskip1pt\mathrm{D}\hskip1.5pt-\,\hh\xi\otimes\mathrm{Id}\hh$, 
our assertion will follow if we show that the pair 
$\,(\mathrm{D}\hh,\hs\xi)\,$ satisfies condition (a) in Lemma~\ref{lieor}. To 
this end, let us assume that, on the contrary, (a) in Lemma~\ref{lieor} fails 
to hold. By Lemma~\ref{lieor}, 
$\,\mathrm{D}\hs\xi=\hs\om-\hh c\hskip1.5pt\xi\otimes\xi\,$ for some constant 
$\,c\,$ and a parallel nonzero $\,2$-form $\,\om$. Since $\,\om\,$ is 
parallel, $\,\om=\mathrm{D}\eta\,$ for the $\,1$-form $\,\eta\,$ given by 
$\,\eta_y(v)=\om(y-o,v)$, for $\,y\in U\subset\plane\,$ and 
$\,v\in T_y\plane$, where $\,o\,$ is any fixed origin in the af\-fine plane 
$\,\plane$. Hence $\,\mathrm{D}\hh(\xi-\eta)=-\hs c\hskip1.5pt\xi\otimes\xi\,$ 
is symmetric, that is, $\,\xi-\eta=d\hs\varphi\,$ for some function 
$\,\varphi$, and the $\,3$-ten\-sor 
$\,-\hs c\hs\mathrm{D}\hh(\xi\otimes\xi)\hs
=\hs\mathrm{D}\hh\mathrm{D}\hs d\hs\varphi$ is totally symmetric. 
It follows now that $\,c=0$. Namely, if we had $\,c\ne0$, the 
resulting symmetry of 
$\,[\hs\mathrm{D}\hh(\xi\otimes\xi)](u,v,\,\cdot,)
=\mathrm{D}_u[\hs\xi(v)\hs\xi]\,$ in the $\,\mathrm{D}\hs$-par\-al\-lel vector 
fields $\,u,v$, combined with the equality 
$\,\mathrm{D}\hs\xi=\hs\om-\hh c\hskip1.5pt\xi\otimes\xi\,$ and (\ref{dxi}.b), 
would give $\,3\hs\xi\otimes\om=\xi\wedge\om$, while $\,\xi\wedge\om=0\,$ as 
$\,\dim\hs\plane=2$, and so $\,\xi\otimes\om=0$, which is impossible, since 
$\,\xi\ne0\,$ and $\,\om\ne0\,$ by Lemma~\ref{lieor}(b),\hs(d). This 
contradiction shows that $\,c=0\,$ and $\,\xi-\eta\,$ is 
$\,\mathrm{D}\hs$-par\-al\-lel. The way $\,\eta\,$ depends on the origin 
$\,o$ shows that a suitable choice of $\,o\hs$ will yield $\,\xi=\eta$. We 
may thus treat $\,\plane\,$ as a two\hh\diml\ real vector space in which the 
new origin $\,o\hs\,$ is the zero vector, and $\,\xi\,$ is defined as in 
Example~\ref{hlfpl}. Let us now choose a one\hh\diml\ vector sub\-space 
$\,\line\,$ of $\,\plane\,$ such that 
$\,\,U\subset\plane\smallsetminus\line$. (Note that $\,o\notin U\nh$, as 
$\,\xi\ne0\,$ everywhere in $\,\,U\nh$, while $\,\xi_o=0$.) According to 
Example~\ref{hlfpl}, condition (a) in Lemma~\ref{lieor} is actually satisfied, 
contrary to what we assumed earlier in this paragraph. This new contradiction 
completes the proof.
\end{proof}

\section{Nor\-den's theorems}\label{nthm}
Given a connection $\,\nabla\nh\,$ in a vector bundle $\,\bv$, a vector 
sub\-bun\-dle $\,\bl\subset\bv\,$ is called $\,\nabla\nnh${\it-par\-al\-lel\/} 
when it has the property that, for any vector field $\,v\,$ tangent to the 
base manifold, if $\,\psi\,$ is a section of $\,\bl$, so is 
$\,\nabla_{\hskip-2.2ptv}\psi$.

Let $\,\nabla\,$ be a connection in a vector bundle $\,\bv\,$ over a 
manifold $\,\bs$. We say that {\it the projectivization of\/} $\,\nabla\nh\,$ 
{\it is flat\/} if every one\hh\diml\ vector sub\-space of the fibre 
$\,\bv\hskip-1pt_y$ at any point $\,y\in\bs\,$ is the fibre at $\,y\,$ of 
some $\,\nabla\nh$-par\-al\-lel line sub\-bun\-dle of the restriction of 
$\,\bv\,$ to a neighborhood of $\,y$. This is clearly the same as requiring 
flatness (integrability of the horizontal distribution) of the connection 
induced by $\,\nabla\hn\,$ in the bundle of real projective spaces obtained by 
projectivizing $\,\bv$.
\begin{lem}\label{abspa}A connection\/ $\,\nabla\nh\,$ in a 
real\hs/\nh com\-plex vector bundle\/ $\,\bv\,$ over a manifold\/ $\,\bs\,$ is 
projectively flat in the sense of Section\/~{\rm\ref{pfco}} if and only if the 
projectivization of\/ $\,\nabla\nh\,$ is flat.
\end{lem}
\begin{proof}If $\,\nabla\hn\,$ is projectively flat, its curvature tensor 
equals $\,\rho\hh\otimes\text{\rm Id}\,$ for some exact $\,2$-form $\,\rho\,$ 
(see Remark~\ref{prjfl}) and we may choose a $\,1$-form $\,\xi\,$ with 
$\,\rho=d\hs\xi$. By Lemma~\ref{prffl}, the connection $\,\mathrm{D}\,$ 
defined as in (\ref{dtn}) is flat. Let $\,\,U\,$ be a contractible 
neighborhood of any given point of $\,\bs$. Since any 
$\,\mathrm{D}\hs$-par\-al\-lel nonzero section of $\,\hs\bv$ over $\,\,U\,$ 
spans a $\,\nabla\nh$-par\-al\-lel line sub\-bun\-dle, the projectivization of 
$\,\nabla\hn\,$ is flat.

Conversely, let the projectivization of $\,\nabla\hn\,$ be flat. If a section 
$\,\psi\,$ of $\,\bv\,$ defined on an open set $\,\,U\subset\bs\,$ spans a 
$\,\nabla\nh$-par\-al\-lel line sub\-bun\-dle, that is, 
$\,\nabla_{\hskip-2.2ptv}\psi=\xi(v)\hs\psi$ for some $\,1$-form $\,\xi\,$ 
on $\,\,U\,$ and all vector fields $\,v\,$ on $\,\,U\nh$, (\ref{cur}) and 
(\ref{dxi}.c) give 
$\,R\hh(\nh u,v)\hs\psi=[\hs(d\hs\xi)(v,u)]\hs\psi$. Thus, for fixed vectors 
$\,u,v\,$ tangent to $\,\bs\,$ at a point $\,y$, {\it every\/} 
$\,\psi\in\bv_y\smallsetminus\{0\}\,$ is an eigenvector of the operator 
$\,R\hh(\nh u,v):\bv_y\to\bv_y$, which, consequently, is a multiple of the 
identity. Hence $\,\nabla\hn\,$ is projectively flat.
\end{proof}
The following immediate consequence of Lemmas~\ref{rctoc} and~\ref{abspa} was 
first proved by Nor\-den \cite{norden-48}, \cite[\S89, formula (5)]{norden-50} 
for tor\-sion\-free connections.
\begin{thm}\label{nrthm}For a connection\/ $\,\nabla\nh\,$ in the tangent 
bundle of a surface, the projectivization of\/ $\,\nabla\,$ is flat if and 
only if the Ric\-ci tensor of\/ $\,\nabla\,$ is skew-sym\-met\-ric.
\end{thm}
The next result is also due to Nor\-den \cite[\S49]{norden-50}:
\begin{thm}\label{hidim}If\/ $\,\nabla\,$ is a tor\-sion\-free connection in 
the tangent bundle of a manifold of dimension $\,n>2\,$ and the 
projectivization of\/ $\,\nabla\hn\,$ is flat, then\/ $\,\nabla\nh\,$ itself 
is flat.
\end{thm}
\begin{proof}Let $\,R\,$ be the curvature tensor of $\,\nabla$, and let 
$\,u,v,w\,$ be tangent vectors at any point. If $\,u,v\,$ are linearly 
dependent, $\,R\hh(\nh u,v)\hh w=0$. If they are not, choosing $\,w\,$ such 
that $\,u,v,w\,$ are linearly independent, we have 
$\,\rho\hs(\nh u,v)\hh w+\rho\hs(v,w)\hh u+\rho\hs(w,u)\hh v=0\,$ from the 
first Bianchi identity for $\,R=\rho\hh\otimes\text{\rm Id}\,$ (cf.\ 
Lemma~\ref{abspa}), so that $\,\rho\hs(\nh u,v)=0\,$ and, again, 
$\,R\hh(\nh u,v)\hh w=\rho\hs(\nh u,v)\hh w=0\,$ for {\it every\/} vector 
$\,w$.
\end{proof}
Some more results of Nor\-den's are discussed at the end of the next section.

The conclusion of Theorem~\ref{hidim} fails, in general, for connections with 
torsion: to obtain a counterexample, we set 
$\,\nabla\hs=\hskip2pt\mathrm{D}\hskip2pt-\,\hh\xi\otimes\mathrm{Id}$, where 
$\,\mathrm{D}\,$ is a flat connection in the tangent bundle of a manifold 
$\,\bs\,$ of any dimension, and the $\,1$-form $\,\hs\xi\hs\,$ on $\,\hs\bs$ 
is not closed. In fact, by Lemmas~\ref{abspa} and~\ref{prffl}, the 
projectivization of $\,\nabla\hn\,$ is flat, while (\ref{rdr}) shows that 
$\,\nabla\hn\,$ itself is not flat.

Counterexamples as above cannot contradict Theorem~\ref{hidim} by producing 
a {\it tor\-sion\-free\/} connection $\,\nabla\,$ in any dimension $\,n>2$. 
In fact, if the torsion tensor $\,\T\,$ of a flat connection $\,\mathrm{D}\,$ 
in $\,\tb\,$ equals $\,\xi\wedge\mathrm{Id}\hh$, for a $\,1$-form $\,\xi\,$ 
(and so $\,\nabla\hs=\hs\mathrm{D}\hskip2pt-\,\hh\xi\otimes\mathrm{Id}\,$ is 
tor\-sion\-free), then, expressing $\,[\hskip2.5pt,\hskip1pt]\,$ in terms of 
$\,\nabla\,$ and $\,\T\,$ we obtain $\,[u,v]=\xi(v)\hh u-\xi(u)\hh v\,$ and 
$\,\xi([u,v])=0\,$ for any $\,\mathrm{D}\hs$-par\-al\-lel vector fields 
$\,u,v$. The Jacobi identity and (\ref{dxi}.c) now give 
$\,(d\hs\xi)(\nh u,v)\hh w+(d\hs\xi)(v,w)\hh u+(d\hs\xi)(w,u)\hh v=0$ 
whenever $\,u,v,w\,$ are $\,\mathrm{D}\hs$-par\-al\-lel, which, as in the 
proof of Theorem~\ref{hidim}, implies that $\,d\hs\xi=0\,$ when $\,n>2$.

\section{Frac\-tion\-al-lin\-e\-ar La\-grang\-i\-ans and first 
integrals}\label{lagr}
Among tor\-sion\-free connections $\,\nabla\hn\,$ on surfaces, those with 
skew-sym\-met\-ric Ric\-ci tensor have an interesting characterization in 
terms of the geodesic flow, discovered by An\-der\-son and Thomp\-son 
\cite[pp.\ 104--107]{anderson-thompson}. It consists, locally, in the 
existence of a frac\-tion\-al-lin\-e\-ar La\-grang\-i\-an for which the 
$\,\nabla\nh$-ge\-o\-des\-ics are the Eu\-ler-La\-grange trajectories. A 
related result of Nor\-den \cite[\S89]{norden-50} establishes the existence of 
a frac\-tion\-al-lin\-e\-ar first integral for the geodesic equation as a 
consequence of skew-sym\-me\-try of the Ric\-ci tensor (cf.\ also 
\cite{norden-46}). The definition of a frac\-tion\-al-lin\-e\-ar function can 
be found in Appendix C.

Both results are presented below. First, what An\-der\-son and Thomp\-son 
proved in \cite[pp.\ 104--107]{anderson-thompson} can be phrased as follows.
\begin{thm}\label{fllag}For a tor\-sion\-free connection\/ $\,\nabla\nh\,$ on 
a surface\/ $\,\bs$, the Ric\-ci tensor\/ $\,\hs\rho$  of\/ $\,\nabla\hn\,$ is 
skew-sym\-met\-ric if and only if every point in\/ 
$\,\tb\smallsetminus\bs\,$ has a neighborhood\/ $\,\,U\nh\,$ with a 
frac\-tion\-al-lin\-e\-ar La\-grang\-i\-an\/ $\,L:U\to\bbR\,$ such that the 
solutions of the Eu\-ler-La\-grange equations for\/ $\,L\,$ coincide with 
those geodesics of\/ $\,\nabla\hn\,$ which, lifted to $\,\tb$, lie in\/ 
$\,\,U\nh$.
\end{thm}
\begin{proof}Let $\,\rho\,$ be skew-sym\-met\-ric. Locally, in $\,\bs$, we may 
choose a $\,1$-form $\,\xi\,$ with $\,\hs d\hs\xi=\rho\,$ and 
$\,\mathrm{D}\hs$-par\-al\-lel vector fields $\,v,w\,$ trivializing $\,\tb$, 
where $\,\mathrm{D}\,$ is the flat connection given by 
$\,\mathrm{D}\hs=\hs\nabla\hs+\,\hh\xi\otimes\mathrm{Id}\,$ (see 
Corollary~\ref{dxtna}). Since the $\,\nabla\nh$-ge\-o\-des\-ic equation 
$\,\nabla_{\hskip-2.2pt\dot y}\dot y=0\,$ for a curve 
$\,t\mapsto y(t)\in\bs\,$ now amounts to 
$\,\mathrm{D}_{\dot y}\dot y=\xi(\dot y)\hs\dot y$, while, as 
$\,\nabla\nh\,$ is tor\-sion\-free, Remark~\ref{torfm} states that the torsion 
$\,1$-form $\,\tau\,$ of $\,\mathrm{D}\,$ coincides with $\,\xi$, the required 
La\-grang\-i\-an $\,L\,$ exists in view of Theorem~\ref{flham} in Appendix C.

Conversely, if such a La\-grang\-i\-an exists, it has, locally, the form 
$\,\hs L=\eta/\nh\zeta$ appearing in Theorem~\ref{flham}, and so, by 
Theorem~\ref{flham}, the Eu\-ler-La\-grange trajectories for $\,L\,$ (that is, 
the $\,\nabla\nh$-ge\-o\-des\-ics) are characterized by 
$\,\mathrm{D}_{\dot y}\dot y=\tau(\dot y)\hs\dot y$, where $\,\mathrm{D}\,$ is 
a flat connection and $\,\tau\,$ is the torsion $\,1$-form of $\,\mathrm{D}$. 
The connections $\,\hs\nabla$ and 
$\,\mathrm{D}\hs-\,\hh\tau\otimes\mathrm{Id}\,$ thus have the same geodesics, 
and are both tor\-sion\-free according to Remark~\ref{torfm}, so that they 
coincide due to the coordinate form of the geodesic equation. Thus, 
$\,\nabla\,$ has skew-sym\-met\-ric Ric\-ci tensor by Corollary~\ref{dxtna}.
\end{proof}
The next theorem, proved by Nor\-den \cite[\S89]{norden-50} in the 
tor\-sion\-free case, remains valid for connections with torsion. The 
frac\-tion\-al-lin\-e\-ar first integral for the geodesic equation, appearing 
in Nor\-den's original statement, arises as the composite, in which 
$\,\zy$, viewed as a function $\,\tb\to\bbC\,$ real-lin\-e\-ar on the fibres, 
is followed by a frac\-tion\-al-lin\-e\-ar function defined on an open subset 
of $\,\bbC\hs$.
\begin{thm}\label{flfig}Let\/ $\,\nabla\hn\,$ be a connection on a surface 
$\,\bs$. If the Ric\-ci tensor\/ $\,\rho\,$ of\/ $\,\hs\nabla$ is 
skew-sym\-met\-ric, then every point of\/ $\,\bs\,$ has a neighborhood\/ 
$\,\,U\nh\,$ with a com\-plex-val\-ued $\,1$-form\/ $\,\zy\hn\,$ on\/ 
$\,\,U\hn\,$ such that\/ $\,\zy_y:\tyb\to\bbC\,$ is a real-lin\-e\-ar 
isomorphism for each $\,y\in U\hn\,$ and, for any geodesic $\,t\mapsto y(t)\,$ 
of\/ $\,\nabla\hn\,$ contained in\/ $\,\,U\nnh$, either\/ 
$\,\zy\hh(\dot y)=0\,$ for all\/ $\,t$, or\/ $\,\zy\hh(\dot y)\ne0\,$ for 
all\/ $\hs\,t\,$ and\/ 
$\,\zy\hh(\dot y)/|\hh\zy\hh(\dot y)|\,$ is constant as a function of\/ $\,t$.
\end{thm}
\begin{proof}We may choose, locally, a $\,1$-form $\,\xi\,$ with 
$\,d\hs\xi=\rho\,$ and a com\-plex-val\-ued $\,1$-form $\,\zy\,$ such that 
$\,\mathrm{Re}\hskip2.5pt\zy\,$ and $\,\mathrm{Im}\hskip2.5pt\zy\,$ trivialize 
$\,\tab\,$ and are $\,\mathrm{D}\hs$-par\-al\-lel, for the flat connection 
$\,\mathrm{D}\hs=\hs\nabla\hs+\,\hh\xi\otimes\mathrm{Id}\,$ (see 
Corollary~\ref{dxtna}). The $\,\nabla\nh$-ge\-o\-des\-ic equation for a curve 
$\,t\mapsto y(t)\,$ reads $\,\nabla_{\hskip-2.2pt\dot y}\dot y=0$, that is, 
$\,\mathrm{D}_{\dot y}\dot y=\xi(\dot y)\hs\dot y$. Applying the 
$\,\mathrm{D}\hs$-par\-al\-lel form $\,\zy\,$ to both sides, we may rewrite 
the last equation as 
$\,[\hh\zy\hh(\dot y)]\dot{\,}=\xi(\dot y)\hs\zy\hh(\dot y)$. Since $\,\xi\,$ 
is real-val\-ued, $\,\zy\,$ has all the required properties.
\end{proof}
Finally, Nor\-den also showed in \cite[\S89]{norden-50} that, if $\,\zy\hn\,$ 
is a com\-plex-val\-ued $\,1$-form on a surface $\,\bs\,$ and, for every 
$\,\hs y\in\bs$, the real-lin\-e\-ar operator $\,\,\zy_y:\tyb\to\bbC$ 
\hbox{is an} isomorphism, then the conclusion of Theorem~\ref{flfig} holds for 
{\it some\/} tor\-sion\-free connection $\,\nabla\hn\,$ with 
skew-sym\-met\-ric Ric\-ci tensor on $\,\bs$. This is immediate since 
$\,\zy\,$ is $\,\mathrm{D}\hs$-par\-al\-lel for a unique flat connection 
$\,\mathrm{D}$. Setting 
$\,\hs\nabla\hs=\hs\mathrm{D}\hs-\,\hh\tau\otimes\mathrm{Id}$, 
where $\,\tau\,$ is the torsion $\,1$-form of $\,\mathrm{D}$, we easily obtain 
our assertion using Corollary~\ref{dxtna}, Remark~\ref{torfm} and the proof of 
Theorem~\ref{flfig}.

\section{Walker metrics and pro\-ject\-a\-bil\-i\-ty}\label{wmap}
Suppose that $\,\mathcal{V}\,$ is a null parallel distribution on a \psr\ 
manifold $\,(M,g)$. As the parallel distribution $\,\mathcal{V}^\perp$ is 
integrable, replacing $\,M\,$ with a suitable neighborhood of any given 
point, we may assume that
\begin{equation}\label{bdl}
\begin{array}{l}
\mathrm{the\ leaves\ of\ 
}\hs\mathcal{V}^\perp\nnh\mathrm{\ are\ all\ 
contractible\ and\ constitute\ 
the\ fibres}\\
\mathrm{of\ \hskip1.44pta\ \hskip1.44ptbundle\ \hskip1.44ptprojection\ }\,\,
\pi:M\to\bs\,\,
\mathrm{\ over\ \hskip1.44ptsome\ \hskip1.44ptmanifold\ }\,\bs\hh.
\end{array}
\end{equation}
Let a null parallel distribution $\,\mathcal{V}\,$ on a \psr\ manifold 
$\,(M,g)\,$ satisfy (\ref{bdl}) and the additional curvature condition
\begin{equation}\label{rvu}
R\hh(v,\,\cdot\,)\hs u=0\hskip9pt\mathrm{for\ all\ sections}\hskip5ptv
\hskip5pt\mathrm{of}\hskip5pt\mathcal{V}\hskip4pt\mathrm{and}\hskip5ptu
\hskip5pt\mathrm{of}\hskip5pt\mathcal{V}^\perp.
\end{equation}
Then, by \cite[p.\ 587, 
assertions (ii) and (iv) in Section 14]{derdzinski-roter},
\begin{enumerate}
  \def\theenumi{{\rm\alph{enumi}}}
\item[(a)] the requirement that $\,\pi^*\xi=g(v,\,\cdot\,)\,$ defines a 
natural bijective correspondence between sections $\,v\,$ of $\,\mathcal{V}\,$ 
parallel along $\,\,\mathcal{V}^\perp$ and sections $\,\xi\nh\,$ of $\,\tab$,
\item[(b)] there exists a unique tor\-sion\-free connection $\,\nabla\hn\,$ on 
$\,\bs\,$ such that, for any $\,\pi$-pro\-ject\-a\-ble vector field $\,w\,$ on 
$\,M$, if $\,v\,$ and $\,\xi\,$ realize the correspondence in (a), then so do 
$\,v\hh'\nh=\hbox{\hskip1.3pt$\overline{\hskip-1.3pt\nabla\hskip-1.3pt}$}_
{\hskip-1.2ptw}v\,$ and 
$\,\xi\hs'\nh=\nabla_{\hskip-1.4pt\pi\hskip-.4ptw}\hs\xi$, where the vector 
field $\,\hs\pi w\hs\,$ on $\,\hs\bs$ \hbox{is the} $\,\pi$-im\-age of $\,w$, 
and 
$\,\hbox{\hskip1.3pt$\overline{\hskip-1.3pt\nabla\hskip-1.3pt}$\hskip1.3pt}\,$ 
denotes the Le\-vi-Ci\-vi\-ta connection of $\,g$.
\end{enumerate}
We refer to $\,\nabla\hn\,$ as the {\it projected connection\/} on $\,\bs$, 
corresponding to $\,g\,$ and $\,\bv$.

\section{A theorem of D\'\i az\hh-Ramos, Garc\'\i a\hh-R\'\i o and 
V\'azquez\hh-Lorenzo}\label{drgr}
One says that an en\-do\-mor\-phism of a pseu\-\hbox{do\hs-}Euclid\-e\-an 
$\,3$-space is of {\it Pe\-trov type\/} III if it is self-ad\-joint and sends 
some ordered basis $\,(X,Y,Z)\,$ to $\,(0,X,Y)$. We are interested in the case 
where this en\-do\-mor\-phism is the self-du\-al Weyl tensor of an oriented 
\psr\ four\mfd\ $\,(M,g)\,$ of the neutral signature $\,(\mmpp)$, acting in 
the $\,3$-space of self-du\-al $\,2$-forms at a point of $\,M$.

In \cite[Theorem~3.1(ii.3)]{diaz-ramos-garcia-rio-vazquez-lorenzo} 
D\'\i az\hh-Ramos, Garc\'\i a\hh-R\'\i o and V\'azquez\hh-Lorenzo 
described the lo\-cal-i\-som\-e\-try types of all those 
cur\-va\-ture\hskip.5pt-ho\-mo\-ge\-ne\-ous self-du\-al oriented 
Ein\-stein four\hs-man\-i\-folds of the neutral metric signature $\,(\mmpp)$, 
which are of Pe\-trov type III, in the sense that so is the self-dual Weyl 
tensor at each point of the manifold, and admit a two\hh\diml\ null parallel 
distribution. 

The metrics mentioned above can also be characterized as {\it the type\/} III 
\hbox{\it Jor\-dan\hs-} {\it Osser\-man Walker metrics in dimension four}. See 
\cite[Remark 2.1]{diaz-ramos-garcia-rio-vazquez-lorenzo}.

The description in 
\cite[Theorem~3.1(ii.3)]{diaz-ramos-garcia-rio-vazquez-lorenzo} had the 
form of a lo\-cal-co\-or\-di\-nate expression. We rephrase it below (see 
Theorem~\ref{maith}) using co\-or\-di\-nate-free language. The first part 
of Theorem~\ref{maith} slightly generalizes a result of 
Garc\'\i a\hh-R\'\i o, Kupeli, V\'azquez\hh-Abal and V\'azquez\hh-Lorenzo 
\cite[Theorem~9]{garcia-rio-kupeli-vazquez-abal-vazquez-lorenzo}, which 
also uses a co\-or\-di\-nate-free formula and assumes that, in our notation, 
$\,\lambda=0$.

First we need some definitions. Let $\,M=\tab\,$ be the \ts\ of the \ctb\ of 
a \mf\ $\,\bs$, and let $\,\pi:\tab\to\bs\,$ be the bundle projection. Any 
connection $\,\nabla\hn\,$ on $\,\bs\,$ gives rise to the 
Patter\-son\hs-\nh Walker {\it Riemann extension metric\/} 
\cite{patterson-walker}, which is the \prc\ $\,g\nnh^\nabla$ on $\,\tab\,$ 
defined by requiring that all vertical and all $\,\nabla$-hor\-i\-zon\-tal 
vectors be $\,g\nnh^\nabla\nnh$-null, while 
$\,g_x\hskip-4.3pt^\nabla\hskip-2pt(\xi,w)=\xi(d\pmb_xw)\,$ for any 
$\,x\in\tab=M$, any vertical vector $\,\xi\in\kerd\pi_x=\tayb$, with 
$\,y=\pi(x)$, and any $\,w\in T\hskip-2pt_x\hskip-.9ptM$. Patter\-son and 
Walker also studied in \cite{patterson-walker} metrics of the form 
$\,g\,=\,g\nnh^\nabla\nnh+\hh\pi^*\nnh\lambda$, where $\,\lambda\,$ is any 
fixed twice-co\-var\-i\-ant symmetric tensor field on $\,\bs$.
\begin{thm}\label{maith}Let there be given a surface\/ $\,\bs$, a 
tor\-sion\-free connection\/ $\,\nabla\hn\,$ on $\,\bs$ such that the 
Ric\-ci tensor\/ $\,\rho\hn\,$ of\/ $\,\nabla\hn\,$ is skew-sym\-met\-ric and 
nonzero everywhere, and a twice-co\-var\-i\-ant symmetric tensor field\/ 
$\,\lambda\,$ on $\,\bs$. Then, for a suitable orientation of the four\mfd\/ 
$\,M=\tab$, the metric\/ $\,g\,=\,g\nnh^\nabla\hn+\hh\pi^*\nnh\lambda\,$ on\/ 
$\,M$, with the neutral signature\/ $\,(\mmpp)$, is Ric\-ci-flat and 
self-du\-al of Pe\-trov type\/ {\rm III}, the vertical distribution\/ 
$\,\mathcal{V}=\kerd\pi\,$ is $\,g$-null and\/ $\,g$-par\-al\-lel, while\/ 
$\,g\,$ and\/ $\,\bv\,$ satisfy the curvature condition\/ {\rm(\ref{rvu})}, 
and the corresponding projected connection on\/ $\,\bs$, characterized by\/ 
{\rm(b)} in Section\/~{\rm\ref{wmap}}, coincides with our original\/ 
$\,\nabla$.

Conversely, if\/ $\,(M,g)\,$ is a neu\-tral-sig\-na\-ture oriented self-du\-al 
Ein\-stein four\mfd\ of Pe\-trov type\/ {\rm III} admitting a two\hh\diml\ 
null parallel distribution\/ $\,\mathcal{V}$, then, for every $\,x\in M\,$ 
there exist $\,\bs,\nabla,\lambda\,$ as above and a dif\-feo\-mor\-phism of a 
neighborhood of\/ $\,x\,$ onto an open subset of\/ $\,\tab$, under which\/ 
$\,g\,$ corresponds to the metric\/ $\,g\nnh^\nabla\hn+\hh\pi^*\nnh\lambda$, 
and\/ $\,\mathcal{V}\,$ to the vertical distribution\/ $\,\kerd\pi$.
\end{thm}
\begin{proof}In the coordinates $\,y^{\hs j}\nnh,\xi_j$ for $\,\tab\,$ arising 
from a local coordinate system $\,y^{\hs j}$ in $\,\bs$, if we let the 
products of differentials stand for symmetric products and $\,\vg_{\!kl}^j$ 
for the components of $\,\hs\nabla\nnh$, then 
$\,\,g=g\nnh^\nabla\hn+\hh\pi^*\nnh\lambda\hs\,$ 
can be expressed as
\begin{equation}\label{hnp}
g\,=\,2\hs d\hs\xi_j\,dy^{\hs j}
+(\lambda_{kl}-\hs2\hh\xi_j\vg_{\!kl}^jw)\,dy^k\hs dy^{\hh l}.
\end{equation}
Setting $\,x_1=\xi_1$, $\,x_2=\xi_2$, $\,x_3=y^1$, $\,x_4=y^2$, 
$\,\xi=\lambda_{11}$, $\,\eta=\lambda_{22}$, $\,\gamma=\lambda_{12}$, 
$\,P=-2\vg_{\!11}^1$, $\,Q=-2\vg_{\!11}^2$, $\,S=-2\vg_{\!22}^1$, 
$\,T=-2\vg_{\!22}^2$, $\,U=-2\vg_{\!12}^1$ and $\,V=-2\vg_{\!12}^2$, one 
easily sees that (\ref{hnp}) amounts to formula (2.1) in D\'\i az\hh-Ramos, 
Garc\'\i a\hh-R\'\i o and V\'azquez\hh-Lorenzo's paper 
\cite{diaz-ramos-garcia-rio-vazquez-lorenzo}, with $\,(a,b,c)\,$ given by 
formula (3.2) in \cite{diaz-ramos-garcia-rio-vazquez-lorenzo}. 
Skew-sym\-me\-try of $\,\rho\,$ is in turn equivalent to condition (3.3) in 
\cite{diaz-ramos-garcia-rio-vazquez-lorenzo}: the three equations forming 
(3.3) state that $\,\rho_{11}=0$, $\,\rho_{22}=0$, and, respectively, 
$\,\rho_{12}+\rho_{21}=0$. Relation (3.4) in 
\cite{diaz-ramos-garcia-rio-vazquez-lorenzo} is, however, equivalent to {\it 
symmetry\/} of $\,\rho$, so that, under our assumptions about $\,\rho$, (3.4) 
is not satisfied at any point. Our claim now follows from 
\cite[Theorem~3.1(ii.3)]{diaz-ramos-garcia-rio-vazquez-lorenzo} and 
Walker's theorem \cite{walker}, cf.\ \cite[p.\ 062504-7]{derdzinski-roter-06}.
\end{proof}

\setcounter{section}{1}
\renewcommand{\thesection}{\Alph{section}}
\setcounter{thm}{0}
\renewcommand{\thethm}{\thesection.\arabic{thm}}
\section*{Appendix A. Lie subalgebras of 
$\,\mathfrak{sl}\hh(2,\bbR)$}
In a real vector space $\,\plane\,$ with $\,\dim\plane=2$, two\hh\diml\ Lie 
subalgebras $\,\hs\mathfrak{g}$ of $\,\mathfrak{sl}\hh(\plane)\,$ are in 
a bijective correspondence with one\hh\diml\ vector subspaces $\,\line\,$ of 
$\,\plane$. The correspondence assigns to $\,\line\,$ the set 
$\,\mathfrak{g}\,$ of all $\,B\in\mathfrak{sl}\hh(\plane)\,$ which leave 
$\,\line\,$ invariant. The commutant ideal 
$\,[\mathfrak{g},\mathfrak{g}]\,$ then consists of all 
$\,B\in\mathfrak{sl}\hh(\plane)\,$ with 
$\,\line\subset\mathrm{Ker}\hskip2.4ptB$. This is immediate from the 
following well-known fact.
\begin{thm}\label{lsasl}Let\/ $\,\hs\plane\,$ be a two\hh\diml\ real vector space.
\begin{enumerate}
  \def\theenumi{{\rm(\roman{enumi}}}
\item[(i)] No two\hh\diml\ Lie subalgebra of\/ $\,\mathfrak{sl}\hh(\plane)\,$ 
is Abelian.
\item[(ii)] Every two\hh\diml\ Lie subalgebra of\/ 
$\mathfrak{sl}\hh(\plane)$ has a basis\/ $A,B$ with\/ $[A,B\hh]=A$, and 
any such\/ $\,A,B\,$ have the form\/ {\rm(\ref{aww})} in some basis\/ 
$\,w,w\hh'$ of\/ $\,\plane$.
\end{enumerate}
Conversely, if\/ $\,A,B\in\mathfrak{sl}\hh(\plane)\,$ are given by\/ 
{\rm(\ref{aww})} in some basis\/ $\,w,w\hh'$ of\/ $\,\plane$, then 
%they are linearly independent, 
$\,[A,B\hh]=A$, and\/ $\,\mathrm{span}\hs\{A,B\}\,$ is a two\hh\diml\ Lie 
subalgebra of\/ $\,\mathfrak{sl}\hh(\plane)$. 
\end{thm}
\begin{proof}A bi\-lin\-e\-ar form $\,\langle\,,\rangle\,$ in 
$\,\mathfrak{sl}\hh(\plane)\,$ defined by 
$\,2\langle A,B\rangle=\mathrm{tr}\,AB$, or, equivalently, 
$\,\langle A,A\rangle=-\det A$, has the Lo\-rentz\-i\-an signature 
$\,(\mpp)$, as one sees using the matrix representation. Also, $\,[A,B\hh]\,$ 
is $\,\langle\,,\rangle$-or\-thog\-o\-nal to $\,A\,$ and $\,B$, for 
$\,A,B\,$ in $\,\mathfrak{sl}\hh(\plane)$, as 
$\,\langle[A,B\hh],A\rangle=\langle A,BA\rangle-\langle BA,A\rangle=0\,$ 
(which is nothing else than bi-in\-var\-i\-ance of the Kil\-ling form 
$\,\langle\,,\rangle$). Thus, the $\,3$-form $\,\mu\,$ on 
$\,\mathfrak{sl}\hh(\plane)\,$ with 
$\,2\hs\mu(A,B,C)=\langle[A,B\hh],C\rangle\,$ is skew-sym\-met\-ric. 
Furthermore, $\,\mathfrak{sl}\hh(\plane)\,$ carries a unique orientation such 
that $\,\mu(A,B,C)=1\,$ for any pos\-i\-tive-o\-rient\-ed 
$\,(\mpp)$-or\-tho\-nor\-mal basis $\,A,B,C\,$ of 
$\,\mathfrak{sl}\hh(\plane)$, since this is the case for the basis
\[
A=\left[\begin{matrix}0&1\\-\nh1&0\end{matrix}\right],\hskip15pt
B=\left[\begin{matrix}0&1\\1&0\end{matrix}\right],\hskip15pt
C=\left[\begin{matrix}1&0\\0&-\nh1\end{matrix}\right]
\]
when $\,\plane=\rto\nnh$. As 
$\,\langle[A,B\hh],\,\cdot\,\rangle=2\hs\mu(A,B,\,\cdot\,)$, the commutator 
operation in $\,\mathfrak{sl}\hh(\plane)$ equals twice the vector product in 
our oriented pseu\-\hbox{do\hs-}Euclid\-e\-an $\,3$-space, with the volume 
form $\,\mu$. Hence $\,[A,B\hh]\ne0\,$ when $\,A,B\,$ are 
linearly independent: completing them to a basis $\,A,B,C$, we get 
$\,\langle[A,B\hh],C\rangle=2\hs\mu(A,B,C)\ne0$. This proves (i).

By assigning to every two\hh\diml\ (non\hs-Abel\-i\-an) Lie subalgebra 
$\,\mathfrak{g}\,$ of $\,\mathfrak{sl}\hh(\plane)\,$ its 
$\,\langle\,,\rangle$-or\-thog\-o\-nal complement $\,\mathfrak{g}^\perp\nnh$, 
which coincides with $\,[\mathfrak{g},\mathfrak{g}]$, we obtain a bijective 
correspondence between the set of such $\,\mathfrak{g}\,$ and the set of all 
$\,\langle\,,\rangle$-null lines (one\hh\diml\ vector sub\-spaces) in 
$\,\mathfrak{sl}\hh(\plane)$. In fact, if $\,A,B\,$ is a basis of 
$\,\mathfrak{g}\,$ and $\,[A,B\hh]=A$, skew-sym\-me\-try of $\,\mu\,$ shows 
that $\,A\,$ is orthogonal to both $\,A$ and $\,B\,$ and hence spans the 
null line $\,\mathfrak{g}^\perp\nnh=[\mathfrak{g},\mathfrak{g}]$. Conversely, 
a vector subspace $\,\mathfrak{g}\,$ of $\,\mathfrak{sl}\hh(\plane)\,$ such 
that $\,\mathfrak{g}^\perp$ is a null line must be a Lie subalgebra, since 
$\,\mathfrak{g}^\perp\nnh\subset\mathfrak{g}$, and a basis $\,A,B\,$ of 
$\,\mathfrak{g}\,$ with $\,A\in\mathfrak{g}^\perp$ has 
$\,\langle[A,B\hh],A\rangle=2\hs\mu(A,B,A)=0$, and so 
$\,[A,B\hh]\in\mathfrak{g}^{\perp\hs\perp}\nnh=\mathfrak{g}$.

That any linearly independent pair $\,A,B\,$ in $\,\mathfrak{sl}\hh(\plane)\,$ 
with $\,[A,B\hh]=A\,$ has the form (\ref{aww}) in some basis $\,w,w\hh'$ of 
$\,\plane\,$ can be seen as follows. We have 
$\,A\in[\mathfrak{g},\mathfrak{g}]=\mathfrak{g}^\perp\nnh$, where 
$\,\mathfrak{g}=\mathrm{span}\hs\{A,B\}$, and so $\,A\,$ is 
$\,\langle\,,\rangle$-null, that is, 
$\,\mathrm{tr}\,A=\det A=0$. In a basis of $\,\plane\,$ containing an 
element of $\,\mathrm{Ker}\hskip2ptA$, the matrix representing 
$\,A\,$ is therefore triangular, with zeros on the diagonal, so that 
$\,A^2\nh=0$, while $\,A\ne0$. Thus, 
$\,A(\plane)\subset\mathrm{Ker}\hskip2ptA\,$ and, as both spaces are 
one\hh\diml, $\,A(\plane)=\mathrm{Ker}\hskip2ptA$. The relation 
$\,[A,B\hh]=A\,$ implies in turn that $\,\mathrm{Ker}\hskip2ptA\,$ invariant 
under $\,B$, and so $\,B\,$ has real characteristic roots. Since 
$\,\mathrm{tr}\,B=0$, the two roots must be nonzero, or else we would have 
$\,\mathrm{Ker}\hskip2ptB=\mathrm{Ker}\hskip2ptA\,$ and, in a basis containing 
an element of $\,\mathrm{Ker}\hskip2ptA$, the matrices of both $\,A\,$ and 
$\,B\,$ would be triangular, with zeros on the diagonal, contradicting the 
linear independence of $\,A\,$ and $\,B$. Thus, $\,B\,$ is 
di\-ag\-o\-nal\-izable, with some nonzero eigenvalues $\,\pm\hs c\,$ such that 
$\,\mathrm{Ker}\hskip2ptA=\mathrm{Ker}\hskip2pt(B+c)$. Choosing a basis 
$\,w,w\hh'$ of $\,\plane\,$ di\-ag\-o\-nal\-iz\-ing $\,\hs B\,$ with 
$\,w\hh'\nh\in\mathrm{Ker}\hskip2ptA$, we may rescale $\,w\,$ so that 
$\,Aw=w\hh'$ (since 
$\,A(\plane)=A(\mathrm{Ker}\hskip2pt(B-c))=\mathrm{Ker}\hskip2ptA$). Applying 
$\,[A,B\hh]=A\,$ to $\,w\,$ we now get $\,c=1/2$, which yields (\ref{aww}), 
proving (b).
\end{proof}

\setcounter{section}{2}
\setcounter{thm}{0}
\section*{Appendix B. Local Lie\hs-group structures}
In this appendix we state and prove Theorem~\ref{latri}, a well-known result, 
included here to provide a convenient reference for the proof of 
Lemma~\ref{linvt}(i).

Given a real\hs/com\-plex \vs\ $\,\hs\x\,$ of sections of a real\hs/com\-plex 
vector bundle $\,\bv\,$ over a \mf\ $\,\bs$, we will say that $\,\x\,$ {\it 
trivializes\/} $\,\bv\,$ if, for every $\,y\in\bs$, the evaluation operator 
$\,\psi\mapsto\psi_y$ is an isomorphism $\,\x\to\bv_y$. This 
amounts to requiring that $\,\dim\hskip1pt\x\,$ coincide with the fibre dimension 
of $\,\hs\bv\hn\,$ and each $\,v\in\x\,$ be either identically zero, or nonzero 
at every point of $\,\bs$. In other words, some (or any) basis of $\,\x\,$ 
should form a trivialization of $\,\bv$.
\begin{thm}\label{latri}Let a Lie algebra\/ $\,\hs\x\hn\,$ of vector fields on 
a simply connected manifold\/ $\,\bs\,$ trivialize its tangent bundle\/ $\,\tb$, 
and let\/ $\,\Psi:\x\to\mathfrak{g}\,$ be any Lie-algebra isomorphism 
between\/ $\,\hs\x\hn\,$ and the Lie algebra\/ $\,\hs\mathfrak{g}\hn\,$ of 
left-in\-var\-i\-ant vector fields on a Lie group\/ $\,G$. Then there exists a 
mapping\/ $\,F:\bs\to G\,$ such that every\/ $\,v\in\x$ is 
$\,F\nh$-pro\-ject\-a\-ble onto\/ $\,\Psi v$. Any such mapping\/ $\,\hs F\nh\,$ 
is, locally, a dif\-feo\-mor\-phism, and\/ $\,\Psi$ determines\/ $\,\hs F\nh\,$ 
uniquely up to compositions with left translations in\/ $\,G$.
\end{thm}
\begin{proof}Given $\,(y,z)\in\bs\times G$, let 
$\,K_{y,z}:\tyb\to T_{(y,z)}(\bs\times G)=\tyb\times T_zG$ \hbox{be the linear} 
operator with $\,K_{y,z}u=(u,(\Psi u')_z)\,$ for $\,u'\in\x\,$ characterized 
by $\,u'_y=u\in\tyb$. Since $\,\Psi u'$ is left\inv, the formula 
$\,\mathcal{H}_{(y,z)}=K_{y,z}(\tyb)$ defines a vector subbundle of 
$\,T(\bs\times G)$, invariant under the left action of $\,G\,$ on 
$\,\bs\times G$. Thus, $\,\mathcal{H}\,$ is (the horizontal distribution of) a 
$\,G$-connection in the trivial $\,G$-principal bundle over $\,\bs\,$ 
with the total space $\,\bs\times G$.

The distribution $\,\mathcal{H}\,$ on $\,\bs\times G\,$ is integrable, that 
is, our $\,G$-connection is flat. In other words, the 
$\,\mathcal{H}$-horizontal lift operation $\,\hs v\hs\mapsto\tilde v$, applied 
to vector fields $\,\hs v,w$ on $\,\bs\,$ is a Lie\hs-algebra homomorphism. In 
fact, $\,\tilde v_{(y,z)}=(v_y,(\Psi v\hh')_z)$, with $\,v'\nh\in\x$ such that 
$\,v\hh'_{\nh y}=v_y$. Choosing in $\,\x\,$ a basis $\,e_j$, $\,j=1,\dots,n$, we 
have $\,v=v^{\hs j}e_j$, $\,w=w^{\hs j}e_j$, 
$\,[\hh e_j,e_k]=c^{\hs l}_{jk}\hs e_{\hh l}$ and 
$\,[\Psi e_j,\Psi e_k]=c^{\hs l}_{jk}\hs\Psi e_{\hh l}$ for some real numbers 
$\,c^{\hs l}_{jk}$ and functions $\,v^{\hs j},w^{\hs j}\nnh$. (The indices 
$\,j,k,l=1,\dots,n$, if repeated, are summed over.) Thus, 
$\,\tilde v=(v,v^{\hs j}\Psi e_j)$, that is, 
$\,\tilde v_{(y,z)}=(v_y,v^{\hs j}(y)(\Psi e_j)_z)$, and similarly for $\,w$. 
Hence $\,[\tilde v,\tilde w]
=([v,w],(d_vw^{\hh l}-d_wv^{\hh l}+v^{\hs j}w^kc^{\hs l}_{jk})\Psi e_{\hh l})$, as 
required: namely, $\,[v,w]=[v^{\hs j}e_j,w^ke_k]
=(d_vw^{\hh l}-d_wv^{\hh l}+v^{\hs j}w^kc^{\hs l}_{jk})\hs e_{\hh l}$, and so 
$\,[v,w]^l=d_vw^{\hh l}-d_wv^{\hh l}+v^{\hs j}w^kc^{\hs l}_{jk}$.

Therefore, as $\,\bs\,$ is simply connected, $\,\bs\times G\,$ is the disjoint 
union of the leaves of $\,\mathcal{H}$, and the projection 
$\,\pi:\bs\times G\to\bs\,$ maps each leaf $\,N\hs$ 
dif\-feo\-mor\-phi\-cal\-ly onto $\,\bs\,$ (cf.\ 
\cite[Vol.\hskip1.9ptI, Corollary~9.2, p.\ 92]{kobayashi-nomizu}). On the 
other hand, one easily sees that a mapping $\,F\,$ has the properties claimed 
in our assertion if and only if $\,d\varXi_y=K_{y,F(y)}$ for all $\,y\in\bs$, 
where $\,\varXi:\bs\to\bs\times G\,$ is given by $\,\varXi(y)=(y,F(y))$. 
Equivalently, $\,\varXi\,$ is required to be an $\,\mathcal{H}$-horizontal 
section of the $\,G$-bundle $\,\bs\times G$, that is, the inverse 
dif\-feo\-mor\-phism $\,\bs\to N\,$ of $\,\pi:N\to\bs\,$ for some leaf 
$\,N\hs$ of $\,\mathcal{H}$. The existence of $\,F\,$ and its uniqueness up 
to left translations are now immediate, while such $\,F\,$ is, locally, a 
dif\-feo\-mor\-phism in view of the inverse maping theorem. This completes the 
proof.
\end{proof}

\setcounter{section}{3}
\setcounter{thm}{0}
\section*{Appendix C. La\-grang\-i\-ans and Ham\-il\-ton\-i\-ans}
\setcounter{equation}{0}
A more detailed exposition of the topics oulined here can be found in 
\cite{mackey}.

We use the same symbol $\,\bv,$ for the total space of a vector bundle 
$\,\bv\,$ over a manifold $\,\bs\,$ as for the bundle itself, identifying 
each fibre $\,\bv_y$, $\,y\in\bs$, with the submanifold $\,\pi^{-1}(y)$ of 
$\,\bv$, where $\,\pi:\bv\to\bs\,$ is the bundle projection. (Thus, $\,\tb$ 
and $\,\tab\,$ are manifolds.) As a set, 
$\,\bv=\{(y,\psi):y\in\bs,\hskip6pt\psi\in\bv_y\}$.

The identity mapping $\,\plane\to\plane\,$ in a real vector space 
$\,\plane\,$ with $\,\dim\plane<\infty$, treated as a vector field on 
$\,\plane$, is called the {\it radial vector field\/} on $\,\plane$. On 
the total space $\,\bv\,$ of any vector bundle over a manifold $\,\bs\,$ 
we have the {\it radial vector field}, denoted here by $\,\mathbf{x}$, 
which is vertical (tangent to the fibres) and, restricted to each fibre 
of $\,\bv$, coincides with the radial field on the fibre.

By a {\it La\-grang\-i\-an\/} $\,L:U\to\bbR\hs$, or, respectively, a 
{\it Ham\-il\-ton\-i\-an\/} $\,\hs H:\ust\nnh\to\bbR$ in a manifold $\,\bs\,$ 
one means a function on a nonempty open set $\,U\subset\tb\,$ or 
$\,\ust\nnh\subset\tb$. The {\it Le\-gendre mapping} $\,\,U\to\tab$, or 
$\,\ust\nnh\to\tb$, associated with $\,L\,$ or $\,H$, is defined by 
requiring that, for each $\,y\in\bs$, it send any $\,v\in U\cap\tyb\,$ or 
$\,\xi\in\ust\nnh\cap\tayb\,$ to the differential of $\,L:U\cap\tyb\to\bbR\,$ 
(or, of $\,H:\ust\nnh\cap\tyb\to\bbR$) at $\,v\,$ (or at $\,\xi$), which is 
an element of $\,T_v^*(U\cap\tyb)=\tayb\subset\tab\,$ or, respectively, 
of $\,T_\xi^*(\ust\nnh\cap\tayb)=\tyb\subset\tb$. We call such a 
La\-grang\-i\-an $\,L:U\to\bbR\,$ or Ham\-il\-ton\-i\-an 
$\,H:\ust\nnh\to\bbR\,$ in $\,\bs\,$ {\it nonsingular\/} if the 
associated Le\-gendre mapping is a dif\-feo\-mor\-phism $\,\,U\to \ust$, 
or $\,U\to \ust$ (then referred to as the Le\-gendre {\it transformation\/}), 
for some open set $\,\,\ust\nnh\subset\tab\,$ or, respectively, 
$\,\,U\subset\tb$. Nonsingular La\-grang\-i\-ans $\,L\,$ in $\,\bs\,$ are 
in a natural bijective correspondence with nonsingular Ham\-il\-ton\-i\-ans 
$\,H\,$ in $\,\bs$. Namely, if $\,L:U\to\bbR\,$ is nonsingular, we define 
$\,H:U\to\bbR\,$ by $\,H=\hs d_{\mathbf{x}}L-L$, for the radial vector 
field $\,\mathbf{x}\,$ mentioned above, and then use the Le\-gendre 
transformation to identify $\,\,U\,$ with $\,\,\ust$, so that 
$\,H\,$ becomes a function $\,\,\ust\nnh\to\bbR\hs$. A nonsingular 
Ham\-il\-ton\-i\-an $\,H:\ust\nnh\to\bbR\,$ similarly gives rise to 
$\,L:\ust\nnh\to\bbR\,$ with $\,L=\hs d_{\mathbf{x}}H-H\,$ that may be 
viewed as a function $\,L:U\to\bbR\hs$. We will write 
$\,L\leftrightarrow H\,$ if $\,L\,$ and $\,H\,$ correspond to each other 
under the assignments $\,L\mapsto H\,$ and $\,H\mapsto L\,$ (easily seen 
to be each other's inverses).

A La\-grang\-i\-an $\,L:U\to\bbR\,$ and a Ham\-il\-ton\-i\-an 
$\,H:\ust\nnh\to\bbR\,$ in $\,\bs\,$ both give rise to {\it equations of 
motion}. For $\,L\,$ these are the {\it Eu\-ler-La\-grange equations}, imposed 
on curves $\,t\mapsto y(t)\in\bs\,$ the velocity of which, viewed as a curve 
$\,t\mapsto v(t)\in\tb$, lies entirely in $\,\,U\nh$, while $\,H\,$ leads to 
{\it Ham\-il\-ton's equations}, imposed on curves 
$\,t\mapsto(y(t),\hs\xi(t))\in\ust$. In the coordinates 
$\,y^{\hs j}\nnh,v^{\hs j}$ for $\,\tb\,$ (or, $\,y^{\hs j}\nnh,\xi_j$ for 
$\,\tab$), induced by a local coordinate system $\,y^{\hs j}$ in $\,\bs$, the 
former read $\,[\hs\partial L/\partial v^{\hs j}]\hs\dot{\,}
=\hs\partial L/\partial y^{\hs j}\nnh$, and the latter 
$\,\dot y^{\hs j}\nh=\hs\partial H\hn/\partial\xi_j$, 
$\,\dot\xi_j=-\hs\partial H\hn/\partial y^{\hs j}\nnh$, with 
$\,(\hskip2.5pt)\hs\dot{\,}=\hh d/dt$. Both systems of equations can be 
rephrased in co\-or\-di\-nate\hs-free terms: the former 
characterizes curves pa\-ram\-e\-trized by closed intervals 
$\,[a,b\hh]\,$ which are fixed-ends critical points of the {\it 
action functional\/} given by $\,\int_a^{\hs b}L(v(t))\,dt$, while the 
latter describes the integral curves of the unique vector field 
$\,X\hskip-1.7pt_H$ on $\,\ust$ with 
$\,\sym(X\hskip-1.7pt_H,\,\cdot\,)=\hh dH$, where $\,\sym\,$ is the 
symplectic form on $\,\tab\,$ (see Remark~\ref{sympl} below). If such 
$\,L\,$ and $\,H\,$ are both nonsingular and $\,L\leftrightarrow H$, 
the Le\-gendre transformation maps the set of solutions of the 
Eu\-ler-La\-grange equations for $\,L\,$ bijectively onto the set of 
solutions of Ham\-il\-ton's equations for $\,H$.

By a {\it frac\-tion\-al-lin\-e\-ar function\/} in a two\hh\diml\ real 
vector space $\,\plane\,$ we mean any rational function of the form 
$\,\eta/\nh\zeta$, defined on a nonempty open subset of 
$\,\plane\smallsetminus\mathrm{Ker}\,\zeta$, where $\,\zeta,\eta\in\plane^*$ 
are linearly independent functionals. Similarly, given a real vector bundle 
$\,\bp\,$ of fibre dimension $\,2\,$ over a manifold, a function 
$\,\,U\to\bbR\,$ on an open set $\,\,U\,$ in the total space $\,\bp\,$ will 
be called {\it frac\-tion\-al-lin\-e\-ar\/} if its restriction to every 
nonempty intersection $\,\,U\cap\bp_{\hskip-2pty}$, for $\,y\in\bs$, is 
frac\-tion\-al-lin\-e\-ar.
\begin{rem}\label{frcli}For $\,\plane,\zeta,\eta\,$ as above, 
$\,d\hh(\nh\eta/\nh\zeta)=\zeta^{-2}(\zeta\hh d\eta-\eta\hskip1ptd\zeta)\,$ is 
easily verified to be a dif\-feo\-mor\-phism 
$\,\plane\smallsetminus\mathrm{Ker}\,\zeta
\to\plane^*\nnh\smallsetminus\mathrm{Ker}\,w\,$ with the inverse 
dif\-feo\-mor\-phism $\,d\hh(\hn v/\hn w)$, where $\,v,w\,$ is the basis of 
$\,\plane\,$ dual to $\,\zeta,\eta$. Note that $\,v/\hn w\,$ then is a 
frac\-tion\-al-lin\-e\-ar function 
$\,\plane^*\nnh\smallsetminus\mathrm{Ker}\,w\to\bbR\hs$.
\end{rem}
\begin{rem}\label{sympl}The total space $\,\tab\,$ of the cotangent bundle of 
any manifold $\,\bs$ carries the symplectic form $\,\sym=\hs d\hh\kappa$, 
where $\,\kappa\,$ is the {\it canonical\/ $\,1$-form}, defined by 
$\,\kappa_\xi(u)=\xi(d\pi_\xi u)\,$ for any 
$\,\xi\in\tab\,$ and $\,u\in T_\xi(\tab)\,$ (so that, at the same time, 
$\,\xi\in\tayb\,$ for $\,y=\pi(\xi)$). In coordinates $\,y^{\hs j}\nnh,\xi_j$ 
as above, $\,\kappa=\xi_j\hh dy^{\hs j}$ and 
$\,\sym=d\hs\xi_j\wedge dy^{\hs j}\nnh$.
\end{rem}
\begin{thm}\label{flham}Given vector fields\/ $\,v,w\,$ trivializing the tangent 
bundle\/ $\,\tb\,$ of a surface $\,\bs$,  let us define a La\-grang\-i\-an\/ 
$\,L:U\to\bbR\,$ and Ham\-il\-ton\-i\-an\/ $\,\hs H:\ust\nnh\to\bbR$ in $\,\bs$, 
both frac\-tion\-al-lin\-e\-ar, by $\,L=\eta/\nh\zeta\,$ and\/ $\,H=v/\hn w$, 
where\/ $\,\zeta,\eta\,$ are the $\,1$-forms dual to $\,v,w\,$ at each point, 
treated as functions $\,\tb\to\bbR\,$ linear on each fibre, and\/ $\,v,w\,$ 
are similarly viewed as functions $\,\tab\to\bbR\,$ linear on the fibres, while 
$\,\,U=\tb\smallsetminus\mathrm{Ker}\,\zeta\,$ and\/ 
$\,\,\ust\nh=\tab\smallsetminus\mathrm{Ker}\,w\,$ are the complements 
in $\,\tb\,$ and\/ $\,\tab\,$ of the total spaces of the line 
sub\-bun\-dles $\,\mathrm{Ker}\,\zeta\,$ and\/ $\,\mathrm{Ker}\,w$. 
Then\/ $\,L,H\,$ are both nonsingular, $\,L\leftrightarrow H\,$ under 
the Le\-gendre transformation, and the solutions of the Eu\-ler-La\-grange 
equations for\/ $\,L\,$ are precisely the curves $\,t\mapsto y(t)\in\bs\,$ 
with
\begin{equation}\label{dyy}
\mathrm{D}_{\dot y}\dot y\,\,=\,\,\tau(\dot y)\hs\dot y\hs, 
\end{equation}
where\/ $\,\mathrm{D}\,$ is the flat connection on $\,\bs\,$ such that\/ 
$\,u,v\,$ are $\,\mathrm{D}\hs$-par\-al\-lel, and\/ $\,\hs\tau\,$ is the 
torsion $\,1$-form of $\,\mathrm{D}$, cf.\ Section~\ref{prel}.
\end{thm}
\begin{proof}Setting $\,P=\tau(v)$, $\,Q=\tau(w)\,$ for the torsion $\,1$-form 
$\,\tau\,$ of $\,\mathrm{D}$, we get $\,\tau=P\zeta+Q\hh\eta$, and so 
(\ref{dxt}) with $\,\nabla=\mathrm{D}$, $\,\xi=\zeta\,$ or $\,\xi=\eta$, and 
$\,\t=\tau\,$ gives $\,d\hs\zeta=Q\hh\eta\wedge\zeta$, 
$\,d\eta=P\zeta\wedge\eta$. Let us identify $\,\tab\,$ (and $\,\tb$) with 
$\,\bs\times\rto$ with the aid of the dif\-feo\-mor\-phism 
$\,\bs\times\rto\nh\to\tab\,$ (or, $\,\bs\times\rto\nh\to\tb$) that sends 
$\,(y,r,s)\,$ to $\,(y,r\hh\zeta_y+s\hh\eta_y)\,$ (or, $\,(y,a,b)\,$ to 
$\,(y,a\hh v_y+b\hh w_y)$). We use the same symbols for differential forms 
(including functions) on the factor manifolds $\,\bs\,$ and 
$\,\rto$ as for their pull\-backs to $\,\tab=\bs\times\rto$ or 
$\,\tb=\bs\times\rto\nnh$. For instance, $\,P,Q,\zeta,\eta\,$ and 
$\,\zeta\wedge\eta\,$ also stand for the pull\-backs of these functions/forms 
from $\,\bs\,$ to $\,\bs\times\rto\nnh$. Similarly, the vector 
fields $\,v,w\,$ on $\,\bs\,$ are also treated as vector fields on 
$\,\bs\times\rto\nnh$, tangent to the $\,\bs\,$ factor. To avoid confusion, we 
will refrain from viewing $\,\zeta,\eta\,$ (or $\,v,w$) as functions on 
$\,\tb\,$ (or $\,\tab$), linear on the fibres, and instead denote those 
functions by $\,r,s\,$ (or, respectively, $\,a,b$), which is consistent with 
our convention, since it means nothing else than treating the coordinate 
functions $\,r,s\,$ or $\,a,b\,$ on the $\,\rto$ factor as functions on 
$\,\bs\times\rto\nnh$.

The vector field $\,Z=av+bw+(aP\nh+b\hs Q)\hs\mathbf{x}\,$ on 
$\,\tb$, with $\,\mathbf{x}\,$ denoting the radial vector field, generates the 
flow of equation (\ref{dyy}) imposed on curves $\,t\mapsto y(t)\in\bs$. In fact, 
writing $\,\dot y=av+bw$, where $\,a,b\,$ are functions of $\,t$, we have 
$\,\tau(\dot y)=aP\nh+b\hs Q$, and so (\ref{dyy}) amounts to 
$\,\dot a=(aP\nh+b\hs Q)\hh a\,$ and $\,\dot b=(aP\nh+b\hs Q)\hh b\,$ (with 
$\,\dot y=av+bw$).

Next, in view of Remark~\ref{frcli}, $\,L=b/a$, $\,\hs H=r/s$, the Le\-gendre 
transformation $\,\,U\nnh\to\ust$ sends $\,(y,a,b)\,$ to $\,(y,r,s)\,$ with 
$\,r=-\hs b/a^2$ and $\,s=1/a$, while $\,L,H\,$ are both nonsingular and 
$\,L\leftrightarrow H$. (Note that $\,d_{\mathbf{x}}L=0\,$ and 
$\,d_{\mathbf{x}}H=0\,$ due to homogeneity of $\,L\,$ and $\,H$.) 
Consequently, the Le\-gendre transformation pushes the radial vector field 
$\,\mathbf{x}\,$ in $\,\tb\,$ and the functions $\,a,b,aP\nh+b\hs Q\,$ forward 
onto $\,-\hs\mathbf{x}\,$ in $\,\tab\,$ and the functions 
$\,1/s,-\hs r/s^2\nnh,\hs\varphi/s^2\nnh$, with $\,\varphi=sP-r\hh Q$. Hence 
it pushes the vector field $\,-Z\,$ (for $\,Z\,$ generating the 
flow of (\ref{dyy})) forward onto $\,s^{-2}(\varphi\hs\mathbf{x}\hh-sv+rw)\,$ 
in $\,\tab$. However, $\,s^{-2}(\varphi\hs\mathbf{x}\hh-sv+rw)\,$ equals 
$\,X\hskip-1.7pt_H$, the unique vector field with 
$\,\sym(X\hskip-1.7pt_H,\,\cdot\,)=\hh dH$. Namely, 
$\,dH=d\hh(r/s)=s^{-2}(s\hs dr-r\hs ds)$, while the canonical $\,1$-form 
$\,\kappa\,$ on $\,\tab=\bs\times\rto$ can be expressed as 
$\,\kappa=r\hh\zeta+s\hh\eta$, and so the symplectic form 
$\,\sym=d\hh\kappa\,$ is given by 
$\,\sym=\varphi\,\zeta\wedge\eta-\zeta\wedge dr-\eta\wedge ds$, with 
$\,\varphi=sP-r\hh Q$. Due to the invariance of the solutions of (\ref{dyy}) 
under the parameter reversal, the distinction between $\,-Z\,$ and $\,Z\,$ 
is of no significance, and our assertion follows.
\end{proof}

\end{document}